\documentclass[letterpaper,10pt]{article}
\usepackage[utf8]{inputenc}
\usepackage{amsmath}
\usepackage{amssymb}
\usepackage{amsthm}
\usepackage{enumerate}
\usepackage{colonequals}
\usepackage{color}
\usepackage{hyperref}
\usepackage{datetime}
\usepackage{graphicx}
\newcommand{\inner}[2]{\langle#1,#2\rangle}
\newcommand{\norm}[1]{\left\lVert#1\right\rVert}
\newcommand{\abs}[1]{\left\lvert#1\right\rvert}
\newcommand{\bbR}{\mathbb R}

\DeclareMathOperator{\vol}{vol}

\DeclareMathOperator{\reg}{reg}

\newcommand{\measurerestr}{%
  \,\raisebox{-.127ex}{\reflectbox{\rotatebox[origin=br]{-90}{$\lnot$}}}\,%
}

\newtheorem{thm}{Theorem}[section]
\newtheorem{propn}[thm]{Proposition}
\newtheorem{lemma}[thm]{Lemma}
\newtheorem{cor}[thm]{Corollary}
\theoremstyle{definition}
\newtheorem{rmk}[thm]{Remark}
\newtheorem{defn}[thm]{Definition}
\newtheorem*{notn}{Notation}

\title{The index of shrinkers of the mean curvature flow}
\author{Zihan Hans Liu}
\date{}
\begin{document}
\maketitle
\begin{abstract}
We introduce a notion of index for shrinkers of the mean curvature flow. We then prove a gap theorem for the index of rotationally symmetric immersed shrinkers in $\bbR^3$, namely, that such shrinkers have index at least 3, unless they are one of the stable ones: the sphere, the cylinder, or the plane. We also provide a generalization of the result to higher dimensions.
\end{abstract}
\section{Introduction}
Self-similar shrinkers of the mean curvature flow are submanifolds of $\bbR^{n+1}$ which satisfy the equation
\begin{equation*}
\vec H + \frac 12x^\perp = 0\text.
\end{equation*}
They arise naturally as tangent flows at singularities of the mean curvature flow: If the spacetime point $(x_0, t_0)$ is a singularity of a mean curvature flow, then (subsequential) limits of parabolic rescalings of the flow around this point will converge to a shrinker; see \cite{huisken} and \cite{ilmanensingularities}. For this reason, the stability (or instability) of those singularities is related to the stability of shrinkers. 

Self-shrinkers are stationary with respect to the Gaussian integral
\begin{equation*}
F_{0,1}(\Sigma) = (4\pi)^{-n/2} \int_\Sigma e^{-\abs{x}^2/4}\text,
\end{equation*}
and in fact, by Huisken's monotonicity formula \cite{huisken}, this integral is non-increasing under a rescaled version of mean curvature flow. 

If one na\"ively considers stability with respect to this functional, all shrinkers turn out to be unstable. Indeed, if $\Sigma_s$ is a one-parameter family of translations or dilations of $\Sigma$, with $\Sigma_0 = \Sigma$, then $s \mapsto F(\Sigma_s)$ has a local maximum at $s=0$; see \cite{cm} or \cite{rigger}. But our real interest is in the stability of mean curvature flow singularity, for which the shrinker is a blow-up limit. In that context, translations and dilations of the shrinker correspond to moving the singularity around in space and time without changing its shape. As such, they do not represent true instabilities of the singularity. 

Instead, in \cite{cm}, Colding and Minicozzi introduce an entropy functional, obtained by taking the supremum of the $F$-functional centered at different scales. This entropy is monotonically decreasing under the mean curvature flow. Based on this, they introduce the notion of entropy-stability for shrinkers, which better captures what it means for the corresponding singularity to be stable. They then show that the only entropy-stable self-shrinking hypersurfaces are the sphere, the cylinders, and the plane.

To do so, they define a different stability notion for shrinkers, which they call $F$-stability. Essentially, this is the same as linear stability of the shrinker with respect to the $F$-functional, except that one ``quotients out'' by the linear subspace generated by translations and dilatations of the shrinker. Colding and Minicozzi show that entropy-stability and $F$-stability are closely related, and in fact equivalent except when the shrinker is the isometric product of a line with a lower-dimensional submanifold.

In this work, we introduce the notions of entropy index and $F$-index of a shrinker. Simply put, these count the number of linearly independent variations that demonstrate entropy-instability and $F$-instability in Colding's and Minicozzi's sense. A more precise definition of $F$-index will be given as Definition~\ref{unstabledefn}, and a more precise definition of entropy index as Definition~\ref{entropyindexdefn}. We will first prove results for the $F$-index, and then in Section~\ref{entropyproof}, we will relate these results to the entropy index.

This work focuses on smooth self-shrinking hypersurfaces with a rotational symmetry. The results are easiest to state in dimension 3, where we can give the following as a special case of corollaries to our main theorem:
\begin{thm}
\label{3dversion}Let $\Sigma^2$ be an immersed, oriented shrinker with polynomial volume growth. Suppose that it is rotationally symmetric, but not isometric to the sphere, the cylinder, or the plane. Then it has $F$-index and entropy index at least 3.
\end{thm}

The main theorem itself is valid in any dimension, and we will only need one degree of rotational symmetry. The price is that we need to assume an extra geometric condition:
\begin{thm}\label{maintheorem}
Let $\Sigma^n \subseteq \bbR^{n+1}$ be an immersed, oriented shrinker with polynomial volume growth, and assume that $\Sigma$ is invariant under a subgroup $SO(2) \subset SO(n+1)$ of rotations of the ambient $\bbR^{n+1}$. Assume that $e_r \cdot \mathbf n$ takes on both positive and negative values on~$\Sigma$, where $\mathbf n$ is the unit normal to~$\Sigma$, and $e_r$ is the vector field 
in~$\bbR^{n+1}$ consisting of unit vector pointing away from the ``rotational axis'', i.e. the fixed-point set of the symmetry group~$SO(2)$. Then $\Sigma$ has $F$-index at least~3.
\end{thm}
The additional geometric condition can be understood as ruling out immersed shrinkers which are locally ``radial graphs'' over the fixed-point set of the~$SO(2)$. If we have a full~$SO(n)$ worth of rotational symmetry, then this condition automatically holds:
\begin{cor}
\label{maincor}
Let $\Sigma^n \subseteq \bbR^{n+1}$ be an immersed, oriented shrinker with polynomial volume growth. Assume that $\Sigma$ is invariant under a subgroup $SO(n) \subset SO(n+1)$ of the rotations of the ambient~$\bbR^{n+1}$, 
and that $\Sigma$ is not isometric to either the plane, the cylinder $S^{n-1}(\sqrt{2(n-1)}) \times \bbR$, or the sphere $S^n(\sqrt{2n})$. Then $\Sigma$ has $F$-index at least~3.
\end{cor}
Theorem~\ref{maintheorem} and Corollary~\ref{maincor} only give results for the $F$-index. With a little more work, it will follow that the shrinkers we consider also have entropy index at least 3: 
\begin{cor}
\label{entropycor}
Suppose that $\Sigma$ is a shrinker which satisfies the hypotheses of either Theorem~\ref{maintheorem} or Corollary~\ref{maincor}. Then $\Sigma$ has entropy index at least 3.
\end{cor}
Theorem~\ref{maintheorem} and Corollary~\ref{maincor} will be proved in section~\ref{proofsection}, Corollary~\ref{entropycor} will be proved in section~\ref{entropyproof}. Theorem~\ref{3dversion} is just the special case $n=2$ in Corollaries \ref{maincor}~and~\ref{entropycor}.

In terms of a ``generic mean curvature flow'' suggested in \cite{cm} (see also \cite{coldingminicozzipedersen}), this means that singularities corresponding to those shrinkers are expected to be easier to ``perturb away.'' That is, it should be possible to do so even if one adds up to two dimensions worth of restrictions on what perturbations are permissible. Note, however, that the results here only concern stability in a linearized sense; more work is needed to turn this into a statement about the dynamical stability of singularities of the nonlinear mean curvature flow.

There are many examples of shrinkers satisfying the hypotheses of Corollary~\ref{maincor}. The earliest rigorously constructed ones are the ``Angenent tori'', embedded surfaces diffeomorphic to $S^1 \times S^{n-1}$, found in \cite{angenenttori}. In that work, Angenent also gives numerical examples for several other immersed shrinkers with the same rotational symmetry. In \cite{drugankleene}, Drugan and Kleene construct infinitely many immersed shrinkers in~$\bbR^{n+1}$ with an $SO(n)$ rotational symmetry, and in \cite{drugan}, Drugan constructs a non-embedded rotationally symmetric shrinker homeomorphic to~$S^2$ (different from the numerical examples of Angenent). More recently, in \cite{mcgrath}, McGrath constructs embedded self-shrinkers in $\bbR^{2n}$ with an $SO(n) \times SO(n)$ symmetry. Corollary~\ref{maincor} and Corollary~\ref{entropycor} apply to all of the above examples, and show that they have entropy index at least 3.

In \cite{mollerkleene}, M\o ller and Kleene studied shrinkers satisfying the symmetry condition in Corollary \ref{maincor}, and classify the embedded ones. In fact, the proof of Corollary~\ref{maincor} will be based on this classification.

The main idea behind the proof of Theorem~\ref{maintheorem} is the following simple observation: Under the $SO(2)$ symmetry hypothesis, the stability operator~$L$ turns out to be a separable partial differential operator. We can find a complete set of eigenfunctions of the form $u e^{ik \theta}$, where $\theta$ is the angular variable, and $u$ is independent of~$\theta$. This way, we divide the spectrum of~$L$ into parts corresponding to different values of~$k$, and analyze these ``restricted spectra'' separately. All shrinkers have certain eigenfunctions in the $k=0$ and $k=1$ parts, and when these change sign, we can construct unstable variations in both of the corresponding subspaces.

When $\Sigma$ is compact and does not intersect the fixed-point set of the $SO(2)$ symmetry group, the above outline easily translates into a complete proof. When $\Sigma$ is non-compact, or when $\Sigma$ contains fixed points of the~$SO(2)$, there are more technical difficulties to overcome, a task which takes up the bulk of this paper.

The results in the paper can be compared \cite{urbano}, in which Urbano shows that the index of a minimal surface in $S^3$ is at least five, unless it is totally geodesic, with equality only for the Clifford torus.

\section{Self-shrinkers}
\subsection{Definition and notation}
\begin{defn} A \emph{shrinker} is a smooth, complete, immersed hypersurface~$\Sigma^n$ in~$\bbR^{n+1}$ whose mean curvature vector satisfies $\vec H + \frac 12 x^\perp = 0$, where $x$ is the position vector in~$\bbR^{n+1}$, and $\perp$ denotes orthogonal projection onto the normal bundle of~$\Sigma$. Furthermore, we require that $\Sigma$ have polynomial volume growth, i.e. that $\vol^n(\Sigma \cap B_R(0)) \leq p(R)$ for some polynomial~$p$ (which may depend on~$\Sigma)$.
\end{defn}
\begin{rmk}
In general, one can also consider shrinkers with higher codimension. However, less is known about the stability of these (see e.g.\ \cite{andrewsliwei} and \cite{leelue}), and we will assume that all of our shrinkers have codimension 1.
\end{rmk}
\begin{rmk}In this work, we only consider shrinkers with polynomial volume growth. In practice, this is not very restrictive. For example, Colding and Minicozzi show that any time slice of limit flow of a mean curvature flow starting from a smooth, closed, embedded hypersurface has polynomial (in fact Euclidean) volume growth; see Corollary~2.13 in~\cite{cm}. Furthermore, Ding and Xin (\cite{dingxin}) have shown that any properly immersed hypersurface satisfying the shrinker equation has Euclidean volume growth.
\end{rmk}
\begin{rmk}
All the results in this paper apply to both immersed and embedded shrinkers. However, in terms of notation, we will often pretend that $\Sigma$ is embedded. For example, we may write $\Sigma \cap B_R(0)$ in a situation where $i^{-1}(B_R(0))$ would be more correct, with $i: \Sigma \rightarrow \bbR^{n+1}$ being the immersion map. This will not alter the substance of the statements made; it will always be possible to translate them into the more cumbersome but technically correct notation.\end{rmk}

\begin{defn}If $x_0 \in \bbR^{n+1}$ and $t_0 \in \bbR$, we use $F_{x_0, t_0}$ to denote the \emph{$F$-functional} centered at $x_0$ with timescale $t_0$. It is defined by
\begin{equation*}
F_{x_0, t_0}(\Sigma) = (4\pi t_0)^{-n/2} \int_\Sigma e^{-\abs{x - x_0}^2/(4t_0)} d\vol^n\text.
\end{equation*}
Here, $\Sigma$ is any immersed hypersurface in~$\bbR^{n+1}$.
\end{defn}
\begin{rmk}
Self-shrinkers are precisely the stationary points of the functional~$F_{0,1}$. In fact, more is true: a suitably rescaled version of mean curvature flow is actually the negative gradient flow for the $F$-functional. For more details, see e.g. \cite{huisken}.
\end{rmk}
Colding and Minicozzi also define the \emph{entropy} of a hypersurface:
\begin{defn}
Suppose that $\Sigma$ is an immersed hypersurface in~$\bbR^{n+1}$. Then its \emph{entropy} is
\begin{equation*}
\lambda(\Sigma) = \sup_{x_0 \in \bbR^{n+1}, t_0 > 0} F_{x_0, t_0}(\Sigma)\text.
\end{equation*}
\end{defn}
As a consequence of Huisken's monotonicity formula (\cite{huisken}), the entropy is non-increasing under mean curvature flow. In particular, a mean curvature flow starting from a hypersurface of a certain entropy cannot develop singularities represented by shrinkers of higher entropy.

A shrinker can be seen as a minimal surface in the weighted manifold $(\bbR^{n+1}, e^{-\abs{x}^2/4} d\vol^{n+1})$. It will be convenient to view the shrinker itself as a weighted manifold with the same weight, and use special notation for weighted integrals:
\begin{notn}$ $
\begin{itemize}
\item We will use the following notation for weighted integrals:
\begin{equation*}
[f]_D = (4\pi)^{-n/2} \int_D f e^{-\abs x^2/4} d\vol^n\text,
\end{equation*}
where $D$ is a subset of $\Sigma$ (possibly $\Sigma$ itself) and $f$ is a function on $D$.

\item $L^2_{\mathrm w}(D)$ will denote the corresponding weighted $L^2$-space, i.e. the completion of $C^\infty_c(D)$ with respect to the norm given by $\norm{f}^2 = [f^2]_D$. 
\item If $D$ is open and bounded, we will also use $H^1_{0, \mathrm w}(D)$ to denote the completion of $C^\infty_c(D)$ with respect to the weighted $H^1$-norm given by $\norm{f}^2 = [\abs{\nabla f}^2 + f^2]_D$.

\item $\mathcal L$ will denote the corresponding drift Laplacian, defined by
\begin{equation*}
\mathcal L f = \Delta_\Sigma f - \frac 12 x^T \cdot \nabla f\text,
\end{equation*}
for $f$ a function on $\Sigma$. Here, $(\cdot)^T$ denotes orthogonal projection onto the tangent space of $\Sigma$, and $\cdot$ denotes the inner product in $\bbR^{n+1}$.
\end{itemize}
\end{notn}
The operator~$\mathcal L$ is formally self-adjoint with respect to a weighted inner product:
\begin{propn}
\label{intbp}
Suppose that $f, g \in C^\infty_c(\Sigma)$, where $\Sigma$ is a complete immersed hypersurface in $\bbR^{n+1}$. Then $[-f \mathcal Lg]_\Sigma = [\nabla f \cdot \nabla g]_\Sigma = [-g \mathcal L f]_\Sigma$.
\end{propn}
\begin{proof}
This is a simple consequence of the divergence theorem, using the fact that $\nabla e^{-\abs{x}^2/4} = -\frac 12 x^T e^{-\abs{x}^2/4}$.
\end{proof}
\subsection{Second variation and $F$-stability}
In studying the stability of shrinkers, Colding and Minicozzi introduced the notion of $F$-stability. We will repeat the essentials of this definition below, for easy reference. We will also define the $F$-index, which, roughly speaking, is the dimension of the space of unstable variations, modulo translations and dilatations

Since the rescaled mean curvature is a negative gradient flow for the $F$\-functional, stability of shrinkers should be related to the second variation of this functional. We will begin by stating the second variation formula.
\begin{defn}
Suppose that $i_s: \Sigma \rightarrow \bbR^{n+1}$ is a family of immersions of a hypersurface. We say that $i_s$ is a \emph{variation} of the immersed
hypersurface~$i_0(\Sigma)$. The \emph{variation vector field}~$X$ is given by $X = \partial i_s/\partial s \rvert_{s=0}$. If we can write $X = f\mathbf n$, where $f$ is a scalar function on $\Sigma$ and $\mathbf n$ is the unit normal to~$i_0(\Sigma)$, then we call $f$ the corresponding \emph{variation function}.
\end{defn}
\begin{rmk}We will often use the notation $\Sigma_s = i_s(\Sigma)$ for the variation, and ignore the questions of how the hypersurface~$\Sigma_s$ is parametrized. It will then be understood that we make the parametrization in such a way that the variation vector field is normal to~$\Sigma_0$, something which can always be achieved. The properties we study will be independent of parametrization, so this represents no loss of generality. 
\end{rmk}
\begin{propn}
\label{secondvariationformula}
Let $\Sigma$ be an immersed shrinker of polynomial volume growth. Suppose that $\Sigma_s$ is a compactly supported variation of~$\Sigma$ with variation function~$f$, and that $x_s \in \bbR^{n+1}$ and $t_s \in \bbR^+$ are smooth variations around $x_0 = 0$ and $t_0 = 1$,  where $\tfrac{d}{ds} x_s\rvert_{s=0} = y$, $\tfrac{d}{ds} t_s \rvert_{s=0} = h$. Then
\begin{multline}
\left.\frac{d^2}{ds^2}\right\rvert_{s=0}F_{x_s, t_s}(\Sigma_s)
= \left[-f Lf
 + 2fhH - h^2H^2 + fy \cdot \mathbf n - \frac 12 (y\cdot \mathbf n)^2\right]_\Sigma\text,
\end{multline}
Here, the \emph{stability operator}~$L$, acting on functions on~$\Sigma$, is given by 
\begin{align*}
Lf &= \mathcal L f + \abs{A}^2 f + \frac 12 f \\
   &= \Delta f - \frac 12x^T \cdot \nabla f + \abs{A}^2f +  \frac 12 f\text;
\end{align*}
$\mathbf n$ denotes the normal vector to~$\Sigma$; and $\cdot$ denotes the inner product in~$\bbR^{n+1}$.
\end{propn}
\begin{proof}
See Theorem~4.14 in \cite{cm}.
\end{proof}
It turns out that the stability operator~$L$ always has certain negative eigenvalues (here and in the rest of the article, we will use the convention that $\lambda$ is an eigenvalue for $L$ if $Lu + \lambda u = 0$ for some function $u$).
\begin{propn}
\label{knowneigenfunctions}
Let $\Sigma$ be an oriented shrinker in~$\bbR^{n+1}$, let $H$ denote its mean curvature, and let $\mathbf n$ denote its (consistently oriented) unit normal. Then
\begin{enumerate}[(i)]
\item $LH = H$, and
\item $L(y \cdot \mathbf n) = \frac 12 y \cdot \mathbf n$, for any fixed vector $y \in \bbR^{n+1}$.
\end{enumerate}
\end{propn}
\begin{proof}
See Theorem~5.2 in \cite{cm}.
\end{proof}
These variations correspond to translating and scaling the shrinker. Doing so does lower the value of the $F$-functional. However, as mentioned in the introduction, these variations correspond to just moving a singularity around, so we do not want to consider them true instabilities. A better stability definition, Colding and Minicozzi's $F$-stability from \cite{cm}, ``quotients out'' by these variations.
\begin{defn}
\label{unstabledefn}
Let $\Sigma^n \subset \bbR^{n+1}$ be an immersed shrinker of polynomial volume growth.
\begin{enumerate}[(i)]
\item 
Let $\Sigma_s$ be a compactly supported variation of $\Sigma = \Sigma_0$ with variation function~$f$. We say that the variation is \emph{stable} if there exists some variations $x_s \in \bbR^{n+1}, t_s \in \bbR$ with $x_s = 0$, $t_s = 1$ such that
\begin{equation*}
\left.\frac{d^2}{ds^2}\right\rvert_{s=0} F_{x_s, t_s}(\Sigma_s) \geq 0\text;
\end{equation*}
otherwise we say it is \emph{unstable}.

From Proposition~\ref{secondvariationformula} we see that this property only depends on the variation function~$f$, so we will also say that $f$ is stable or unstable.
\item $\Sigma$ is \emph{$F$-stable} if every compactly supported variation is stable. Otherwise it is \emph{$F$-unstable}.
\item The \emph{$F$-index} of~$\Sigma$ is the maximal dimension of a linear subspace consisting entirely of compactly supported unstable variation functions.
\end{enumerate}
\end{defn}
\begin{rmk}With these definitions, being $F$-stable is equivalent to having $F$-index~0.
\end{rmk}
\section{Rotationally symmetric shrinkers}
\subsection{Definition and notation}
Fix a one-parameter subgroup $G \simeq SO(2)$ of $SO(n+1)$, the group of rotations of~$\bbR^{n+1}$. 
\begin{defn}Let $S$ be any subset of~$\bbR^{n+1}$. We say that $S$ is \emph{rotationally symmetric} if it is invariant under~$G$.

If $\Sigma$ is an immersed hypersurface in~$\bbR^{n+1}$, we say that $\Sigma$ is \emph{rotationally symmetric} if its image is rotationally symmetric as a subset of $\bbR^{n+1}$.
\end{defn}
\begin{notn}$ $
\begin{itemize}
\item $Q$ will denote the quotient~$\bbR^{n+1}/G$. We will identify $Q$ with a half-space~$\bbR^{n-1} \times \bbR_{\geq 0}$.
\item $A$ will denote the fixed-point set of $\bbR^{n+1}$ under the action of~$G$, i.e. $A = \{x \in \bbR^{n+1} \mid gx = x \;\forall g \in G\}$. We will sometimes call this the \emph{symmetry axis}.
\item $(x_1, \ldots, x_{n+1})$ will be used as Cartesian coordinates on~$\bbR^{n+1}$. These are arranged so that $G$ acts by changing only $x_n$ and $x_{n+1}$. $x$ will denote the position vector on~$\bbR^{n+1}$.
\item $(x_1, \ldots, x_{n-1}; r, \theta)$ will be used as cylindrical coordinates on $\bbR^{n+1}$, defined by $x_n = r\cos \theta$ and $x_{n+1} = r \sin \theta$. Thus, for example, $A = \{r = 0\}$.
\item $(x_1', \ldots, x_n')$ will be used as coordinates on $Q = \bbR^{n-1} \times \bbR_{\geq 0}$. Sometimes we will write $r = x_n'$, to emphasize the connections with the cylindrical coordinates on $\bbR^{n+1}$. $x'$ will denote the position vector in~$Q$.
\item $\Sigma'$ will denote the image of $\Sigma$ in $Q$, assuming that $\Sigma$ is rotationally symmetric. Similarly, if $S$ is any rotationally symmetric subset of $\bbR^{n+1}$, then $S'$ will denote its image in~$Q$.
\end{itemize}
\end{notn}
In studying a rotationally symmetric shrinker $\Sigma$, 
it will be helpful to pass to the quotient~$\Sigma'$. We will establish some basic properties of this quotient.

An immersed submanifold will be called \emph{connected} if the domain of the immersion is connected, i.e.\ if it cannot be decomposed non-trivially as two immersed submanifolds.
\begin{lemma}
\label{sigmaprimeprops}
Suppose that $\Sigma$ is a smooth, complete, connected and rotationally symmetric immersed hypersurface in~$\bbR^{n+1}$. 
Then $\Sigma'$ is a smooth immersed hypersurface-with-boundary in~$Q$, whose boundary lies on~$\partial Q$. Furthermore, the interior of~$\Sigma'$ is connected.
\end{lemma}
\begin{proof}
Let $i: \Sigma \rightarrow \bbR^{n+1}$ denote the immersion and let $H$ be some hyperplane in~$\bbR^{n+1}$ containing $A$. Note that $i(\Sigma)$ cannot be tangent to $H$ anywhere; if it were, then, by rotational symmetry, the image of~$i$ would contain an open subset of~$\bbR^{n+1}$ b, contradicting Sard's theorem. It follows that $i$ intersects $H$ transversely, and thus $i(\Sigma) \cap H$ is a smooth, complete, immersed hypersurface in~$H$.

Since $\Sigma'$ can be viewed as the intersection of $i(\Sigma)$ with one half of~$H$, it follows that $\Sigma'$ is (the image of) a smooth, immersed hypersurface in~$Q$, with boundary lying in~$\partial Q$.

It remains to show that the interior of~$\Sigma'$ is connected. First, we will show that $i$ also intersects $A$ transversely. Suppose not. Then there must be a point $x \in i^{-1}(A)$ where $i_\ast(T_x \Sigma)$ is hyperplane $H$ containing $A$. This leads to a contradiction by the same reasoning as in the first paragraph. Thus, $i^{-1}(A)$ is a submanifold of codimension two in $\Sigma$.

Since $\Sigma$ is connected, $\Sigma \setminus i^{-1}(A)$ is connected too, so any two points $p, q \in \Sigma \setminus i^{-1}(A)$ can be joined by a path which avoids $i^{-1}(A)$. Since this path can be pushed down to~$\Sigma'$, this shows that $\Sigma' \setminus \partial \Sigma'$ is connected as well.
\end{proof}
Just as $\Sigma$ is a minimal surface in Euclidean space weighted by $e^{-\abs{x}^2/4}$, $\Sigma'$ is a minimal surface in~$Q$ with the weight $2\pi r e^{-\abs{x'}^2/4}$, and it will be helpful to define notation for the corresponding weighted integrals.
\begin{notn} Suppose that $\Sigma$ is a rotationally symmetric shrinker.
\begin{itemize}
\item We will use the following notation for weighted integrals:
\begin{equation*}
[u]'_D = (4\pi)^{-n/2} \int_{D'} 2\pi u r e^{-\abs{x'}^2/4} d\vol^n\text,
\end{equation*}
where $D'$ is a subset of~$\Sigma'$ and $u$ is a function on~$D'$.
\item $\mathcal L'$ will denote the corresponding drift Laplacian, defined by
\begin{equation*}
\mathcal L' u = \Delta_{\Sigma'} u + \left(\frac 1r e_r^T + \frac 12 (x')^T\right) \cdot \nabla u\text,
\end{equation*}
where $u$ is a function on~$\Sigma'$. Here $(\cdot)^T$ denotes the orthogonal projection onto
the tangent space of~$\Sigma'$, and $\cdot$ denotes the Euclidean inner product in $Q$.
\end{itemize}
\end{notn}
As with~$\mathcal L$, this operator is formally self-adjoint:
\begin{propn}
\label{intbpprime}
Let $\Sigma'$ be an immersed hypersurface in $Q$ with $\partial \Sigma' \subset \partial Q$. Suppose that $f, g \in C^\infty_c(\Sigma')$ (where we allow the possibility that $f$ and $g$ are non-zero on~$\partial \Sigma'$). Then $[-f\mathcal L' g]'_{\Sigma'} = [\nabla f\cdot\nabla g]'_{D'} = [-g \mathcal L' f]'_{D'}$.
\end{propn}
\begin{proof}
The key here is to note that $\nabla(r e^{-\abs{x'}^2/4}) = (e_r^T/r + (x')^T/2) r e^{-\abs{x'}^2/4}$. Therefore,
\begin{equation*}
(\mathcal L' g) r e^{-\abs{x'}^2/4} = \Delta(g r e^{-\abs{x'}^2/4})\text,
\end{equation*}
and by the divergence theorem, we have
\begin{equation*}
[f \mathcal L' g]'_{\Sigma'} + [\nabla f \cdot \nabla g]'_{\Sigma'} = \int_{\partial \Sigma'} f \frac{\partial g}{\partial n} 2\pi r e^{-\abs{x'}^2/4} d \vol^{n-1}\text.
\end{equation*}
The right-hand side vanishes because $r=0$ on $\partial \Sigma'$, and it follows that $[- f\mathcal L' g]'_{\Sigma'} = [\nabla f \cdot \nabla g]'_{\Sigma'}$. By symmetry, we also have $[-g\mathcal L' f]'_{\Sigma'} = [\nabla f \cdot \nabla g]'_{\Sigma'}$.
\end{proof}
\subsection{Fourier modes}
Suppose that $\Sigma$ is a rotationally symmetric immersed hypersurface in~$\bbR^{n+1}$,  that $D$ is a rotationally symmetric open subset of~$\Sigma$ (possibly $\Sigma$ itself), and that $k \geq 0$ is an integer. We will be paying special attention to functions~$f: D \rightarrow \bbR$ of the forms
\begin{align}
\label{cosform}
f(x) &= u(x')\cos(k\theta)\text{, and} \\
\label{sinform}
f(x) &= u(x')\sin(k\theta)\text, 
\end{align}
where $u$ is a function on~$\Sigma'$ only. We will introduce some spaces of functions of this form:
\begin{notn}$ $
\begin{itemize}
\item $C^\infty (k, D)$ consists of smooth functions on~$D$ of the form~\eqref{cosform}.
\item Similarly, $C^\infty_c(k, D)$ and $C^\infty_0(k, D)$ consist of smooth functions of the form~\eqref{cosform} which are compactly supported or which vanish on the boundary of~$D$, respectively.
\item $L^2_{\mathrm w}k, D)$ will denote the completion of~$C^\infty_c(k, D)$ with respect to the weighted $L^2$-norm given by $\norm{f}^2 = [f^2]_D$. When $D$ is open and bounded, we will also use $H^1_{0,\mathrm w}(k,D)$ to denote the completion of~$C^\infty_c(D)$ with respect to the weighted $H^1$-norm given by $\norm{f}^2 = [\abs{\nabla f}^2 + [f^2]_D$. 
\item We define $C^\infty_{\mathrm{sin}}(k, D)$, $C^\infty_{c, \mathrm{sin}}(k, D)$, $C^\infty_{0, \mathrm{sin}}(k, D)$, $L^2_{\mathrm w, \mathrm{sin}}(k, D)$, and $H^1_{0, \mathrm w, \mathrm{sin}}(k, D)$ in the same way, but for functions of the form~\eqref{sinform} instead.
\end{itemize}
\end{notn}
\begin{rmk}
Because of the rotational symmetry, the spaces defined using the sines behave in the same way as the ones using cosines. We will focus on the latter, with the tacit understanding that corresponding results hold for the former as well.
\end{rmk}
Given a function~$f$ on~$\Sigma$, we can decompose it as a Fourier series:
\begin{propn}
\label{fouriersum}
Let $\Sigma$ be a rotationally symmetric immersed hypersurface in~$\bbR^{n+1}$, and let $D$ be a rotationally symmetric subset of~$\Sigma$. Then we can decompose $L^2_{\mathrm w}(D)$ as an orthogonal direct sum of Hilbert spaces
\begin{equation}
\label{l2decomp}
L^2_{\mathrm w}(D) = \bigoplus_{k=0}^\infty L^2_{\mathrm w}(k, D) \oplus \bigoplus_{k=1}^\infty L^2_{\mathrm w, \mathrm{sin}}(k, D)\text.
\end{equation}
\end{propn}
\begin{proof}
If $f \in L^2_{\mathrm w}(D)$, we can view it as a function on~$D' \times S^1$, and decompose it as a Fourier series in the variable $\theta$. Let
\begin{align*}
a_k(x') &= \frac{1}{\pi}\int_{\theta=0}^{2\pi} f(x', \theta) \cos(k\theta) d\theta\quad(k\geq 0) \\
b_k(x') &= \frac{1}{\pi}\int_{\theta=0}^{2\pi} f(x', \theta) \sin(k\theta) d\theta\quad(k \geq 1)\text.
\end{align*}
By Fourier theory, we can write
\begin{equation*}
f(x', \theta) = 2a_0(x') + \sum_{k\geq 1} (a_k(x') \cos(k\theta) + b_k(x') \sin(k\theta))
\end{equation*}
with the convergence in~$L^2_{\mathrm w}(\Sigma)$. This gives the desired decomposition.

The fact that the decomposition is orthogonal simply follows from the orthogonality of the basis functions $\cos(k\theta)$ and $\sin(l\theta)$ in~$L^2(S^1)$. For example, suppose that $f = u(x') \cos(k\theta)$ and $g = v(x') \cos(l\theta)$, with $k \neq l$. Then
\begin{align*}
[fg]_D &= \int_{\theta=0}^{2\pi} [uv]'_{D'} \cos(k\theta) \cos(l\theta) d\theta \\
&= 0.
\end{align*}
because $\cos(l\theta)$ and $\cos(k\theta)$ are orthogonal in~$L^2(S^1)$.
\end{proof}
The decomposition of the previous lemma preserves $H^1$-regularity, a fact that will be helpful for later analysis:
\begin{lemma}
\label{h1decomp}
Let $\Sigma$ be a rotationally symmetric immersed hypersurface in~$\bbR^{n+1}$, and let $D$ be a rotationally symmetric open subset of~$\Sigma$.
If $f$ is any function in $L^2_{\mathrm w}(D)$, we let $f_k$ denote the orthogonal projection of~$f$ onto~$L^2_{\mathrm w}(k, D)$.
\begin{enumerate}[(i)]
\item If $f \in C^\infty(D)$, then $\norm{f_k}_{H^1_{\mathrm w}} \leq \norm{f}_{H^1_{\mathrm w}}$.
\item If $f \in H^1_{0, \mathrm w}(D)$, then $f_k \in H^1_{0,\mathrm w}(D)$.
\end{enumerate}
\end{lemma}
\begin{proof}$ $
\begin{enumerate}[(i)]
\item We can view any function $f(x)$ on~$\Sigma$ as a function $f(x', \theta)$ on $\Sigma' \times S^1$. Under this framework, we have $\abs{\nabla f}^2 = \tfrac{1}{r^2}\abs{\partial f/\partial \theta}^2 + \abs{\nabla_{x'} f}^2$.

$f_k$ is given by
\begin{equation*}
f_k(x', \theta) = \frac 1\pi \int_{\theta' = 0}^{2\pi} f(x', \theta') \cos(k\theta') d\theta' \cos(k\theta)\text.
\end{equation*}
(In the case $k=0$, the above formula is wrong by a factor of $2$. To avoid cumbersome notation, we will assume $k \neq 0$, but the same proof works for that case as well.)

Using this formula, we will estimate $\abs{\nabla f_k}^2$ by decomposing it as $\tfrac{1}{r^2}\abs{\partial_\theta f_k}^2 + \abs{\nabla_{x'} f_k}^2$. For the first term, we have
\begin{align*}
\partial_\theta f_k &= -\frac 1\pi k\sin(k\theta) \int_{\theta' = 0}^{2\pi} f(x', \theta') \cos(k\theta') d\theta' \\
&= \frac 1\pi \sin(k\theta)\int_{\theta'=0}^{2\pi} f(x', \theta') \frac{d}{d\theta'} \sin(k\theta') d\theta' \\
&= -\frac 1\pi \sin(k\theta) \int_{\theta'=0}^{2\pi} \frac{\partial}{\partial \theta'} f(x', \theta') \sin(k\theta') d\theta'\text.
\end{align*}
By Cauchy-Schwarz, it follows that
\begin{align*}
\frac{1}{r^2}\abs{\partial_\theta f_k}^2 &\leq \frac 1{\pi^2 r^2} \sin^2(k\theta) \int_{\theta'=0}^{2\pi} \abs{\partial_{\theta'} f(x', \theta')}^2 d\theta' \int_{\theta'=0}^{2\pi} \sin^2(k\theta') d\theta' \\
&= \frac 1{\pi r^2}\sin^2(k\theta) \int_{\theta'=0}^{2\pi} \abs{\partial_{\theta'} f(x', \theta')}^2 d\theta'
\end{align*}
and after integrating with respect to $\theta$, with $x'$ fixed, we get
\begin{equation}
\label{dthetabound}
\frac{1}{r^2}\int_{\theta=0}^{2\pi} \abs{\partial_\theta f_k}^2 d\theta \leq \frac{1}{r^2}\int_{\theta=0}^{2\pi} \abs{\partial_{\theta} f}^2 d\theta\text.
\end{equation}
As for the term $\abs{\nabla_{x'} f_k}^2$, we use Cauchy-Schwarz directly:
\begin{align*}
\abs{\nabla_{x'} f_k}^2 &= \frac 1{\pi^2} \abs{\int_{\theta'=0}^{2\pi} \nabla_{x'} f(x', \theta') \cos(k\theta') d\theta'}^2 \cos^2(k\theta) \\
&\leq \frac 1{\pi^2} \int_{\theta'=0}^{2\pi} \abs{\nabla_{x'} f(x', \theta')}^2 d\theta' \int_{\theta'=0}^{2\pi} \cos^2(k\theta') d\theta' \cos^2(k\theta) \\
&= \frac 1{\pi} \int_{\theta'=0}^{2\pi} \abs{\nabla_{x'} f(x', \theta')}^2 d\theta' \cos^2(k \theta)\text.
\end{align*}
We integrate with respect to $\theta$, to get
\begin{equation}
\label{dxprimebound}
\int_{\theta=0}^{2\pi} \abs{\nabla_{x'} f_k}^2 d\theta \leq \int_{\theta = 0}^{2\pi} \abs{\nabla_{x'} f}^2 d\theta
\end{equation}
By adding \eqref{dthetabound}~and~\eqref{dxprimebound}, and integrating with respect to $x'$ (with weight $r e^{-\abs{x'}^2/4}$), we get $[\abs{\nabla f_k}^2]_{D} \leq [\abs{\nabla f}^2]_{D}$, and the result follows.
\item This follows from part (i) by an approximation argument.
\end{enumerate}
\end{proof}
Due to smoothness, functions in $C^\infty(k, D)$ are restricted in their possible behavior near the symmetry axis~$A = \{r=0\}$, as the following lemma shows. This will be useful later in dealing boundary terms at~$A$.
\begin{lemma}
\label{specialbehavior}
Suppose that $\Sigma'$ is a rotationally symmetric, smooth, immersed hypersurface in~$\bbR^{n+1}$. Take a function~$f \in C^\infty(k, D)$, and write $f = u(x') \cos(k\theta)$. Let $p' \in \partial \Sigma'$, and let $X' \in T_{p'} \Sigma'$ be the unit normal to~$\partial \Sigma'$.
\begin{enumerate}[(i)]
\item If $k \neq 1$, then $X' \cdot \nabla u$ vanishes at~$p'$. 
\item If $k \geq 1$, then $u = 0$ at~$p'$.
\end{enumerate}
\end{lemma}
\begin{proof}$ $
\begin{enumerate}[(i)]
\item Let $p \in A \subseteq \bbR^{n+1}$ denote the unique preimage of $p'$ under the quotient map~$\bbR^{n+1} \rightarrow Q$. Given an angle~$\phi$, let $X_{\phi} \in T_p \Sigma$ denote the unit vector pointing in the direction given by $\theta = \phi$. Then
\begin{equation}
\label{pushforwardnormal}
X_{\phi} \cdot \nabla f = \cos(k\theta) X' \cdot \nabla u\text.
\end{equation}

When $k \geq 2$, let $\phi_j = 2\pi j/k$ for $1 \leq j \leq k$. We note that $\cos(k\phi_j) = 1$ for all~$j$, and that $\sum_{j=1}^{k} X_{\phi_j} = 0$. Therefore, we obtain the desired result by taking the sum of $k$ instances of \eqref{pushforwardnormal} for these values of~$\phi$.

When $k = 0$, any angle $\phi$ has $\cos(k\phi) = 1$, so we can get the same result by taking the same sum with, for example, $X_0$~and~$X_\pi$.
\item If $k \geq 1$, then the function $\cos(k\theta)$ takes the value~$0$ for some~$\theta$. In other words, $f$ vanishes on certain half-spaces with boundary~$A$, and by continuity must vanish on~$\Sigma \cap A$ as well.
\end{enumerate}
\end{proof}
\subsection{Stability operator}
It turns out that the stability operator~$L$ maps the subsets~$C^\infty(k, \Sigma)$ into themselves. This will be a key point in the analysis of the spectrum of~$L$.
\begin{propn}
\label{lkpropn}Suppose that $\Sigma$ is a rotationally symmetric shrinker, and that $f \in C^\infty(k, \Sigma)$, with $f(x) = u(x')\cos(k\theta)$. Then
\begin{equation*}
Lf = L_k u \cos(k\theta)
\end{equation*}
where the operator~$L_k$, acting on functions on~$\Sigma'$, is given by
\begin{equation}
\label{lkformula}
L_k u = \mathcal L' u + \left(\abs{A}^2 + \frac{1}{r^2}\abs{e_r^\perp}^2 + \frac 12 - \frac{k^2}{r^2}\right)u\text,
\end{equation}
where $^\perp$ denotes orthogonal projection onto the normal bundle of~$\Sigma'$. 

In particular, if $f \in C^\infty(k, \Sigma)$, then $Lf \in C^\infty(k, \Sigma)$.
\end{propn}
\begin{proof}
We can view $f$ as a function~$f(x', \theta)$ on $\Sigma' \times S^1$. The result will follow from the following three identities:
\begin{enumerate}[(i)]
\item $x^T \cdot \nabla^\Sigma f
 = (x')^T \cdot \nabla^{\Sigma'} f $, where $x$ is the position vector in~$\bbR^{n+1}$;
\item $\abs{A_\Sigma}^2 = \abs{A_{\Sigma'}}^2 + r^{-2}\abs{e^\perp_r}^2$; and
\item $\Delta^\Sigma f  = \Delta^{\Sigma'} f + \frac{1}{r^2}\frac{\partial^2}{d\partial^2} f  + \frac 1{r} e_r^T \cdot \nabla^{\Sigma'} f$.
\end{enumerate}
To establish some notation, pick an arbitrary point~$x \in \Sigma$. Thinking of $\Sigma$ as $\Sigma' \times S^1$ (with a warped metric), we write $x = (x', \theta)$, and  let $\Sigma'_x = \Sigma' \times \{\theta\}$ and $\Gamma = \{x'\} \times S^1$ denote the copies of $\Sigma'$~and~$S^1$, respectively, that pass  through~$x$. We will think of these as submanifolds of $\Sigma$, and hence immersed submanifolds of $\bbR^{n+1}$; $\Gamma$ is then just a circle around~$A$ and $\Sigma'_x$ an isometric copy of~$\Sigma'$ which is contained in a half-space with boundary~$A$.

We will show identities (i) to (iii) in turn:
\begin{enumerate}[(i)]
 \item This is immediate, the key point being that the position vector $x$ does not have any component in the $\theta$-direction.
 \item Consider the unit vector~$e_\theta = \frac 1r \frac{\partial}{\partial \theta}$. Then $e_\theta$, thought of as a vector in~$\bbR^{n+1}$, depends only on~$\theta$, and it follows that $A_\Sigma(e_\theta, X) = (\nabla^{\bbR^{n+1}}_X e_\theta)^\perp = 0$ for any vector~$X \in T_p \Sigma'_x$. From this, and the fact that $\Sigma'_x$ is totally geodesic in~$\Sigma$, we conclude that $\abs{A_\Sigma}^2 = \abs{A_{\Sigma'}}^2 + \abs{A_\Sigma(e_\theta, e_\theta)}^2$.
 
 On the other hand, we can easily calculate $A_\Sigma(e_\theta, e_\theta) = (\nabla^{\bbR^{n+1}}_{e_\theta} e_\theta)^\perp = -\frac 1r e_r^\perp$, and identity (ii) follows.
 \item $\Sigma'_x$ is totally geodesic in~$\Sigma$, and it follows that
\begin{equation*}
\Delta_\Sigma u = \Delta_{\Sigma'_x}u + \Delta_\Gamma u - \vec \kappa \cdot \nabla^\Sigma u
\end{equation*}
at~$x$, where $\vec \kappa$ is the geodesic curvature of $\Gamma$ in~$\Sigma$.

The geodesic curvature vector $\kappa$ of~$\Gamma$ in $\Sigma$ is just the orthogonal projection of the curvature vector of~$\Gamma$ in the ambient space $\bbR^{n+1}$ onto the tangent space of~$\Sigma$. It follows that $\vec \kappa = -e_r^T/r$. Furthermore, on the circle $\Gamma$, the Laplacian is given by$\Delta_\Gamma = \frac{1}{r^2}\frac{\partial^2}{\partial \theta^2}$, and (iii) follows.
\end{enumerate}
Substituting in the identities (i), (ii), and (iii) into the formula for~$L$, given in Proposition~\ref{secondvariationformula}, we get the result.
\end{proof}
As the above proposition suggests, the known eigenfunctions of~$L$ from Proposition~\ref{knowneigenfunctions} fit in well in the framework of~$C^\infty(k, \Sigma)$ functions:
\begin{propn}
\label{knownfuncsspecial}
Suppose that $\Sigma$ is a rotationally symmetric shrinker with mean curvature~$H$ and normal vector~$\mathbf n$.
\begin{enumerate}[(i)]
\item The function~$H$ belongs to $C^\infty(0, \Sigma)$.
\item For $i=1,\ldots, n-1$, the functions~$e_i \cdot \mathbf n$ also belong to $C^\infty(0, \Sigma)$.
\item The function~$e_n \cdot \mathbf n$ belongs to $C^\infty(1, \Sigma)$, and the function~$e_{n+1} \cdot \mathbf n$ belongs to $C^\infty_{\mathrm{sin}}(1, \Sigma)$.
\end{enumerate}
\end{propn}
\begin{proof}(i) and (ii) are easy consequences of rotational symmetry. To see that (iii) is true, note that the projection of~$\mathbf n$ onto the~$x_n x_{n+1}$-plane is a vector of magnitude~$\mathbf n_{\Sigma'} \cdot e_r$ and direction parallel to $e_r = e_n \cos \theta + e_{n+1} \sin \theta$. In other words, $\mathbf n \cdot e_n = (\mathbf n_{\Sigma'} \cdot e_r) \cos \theta$, and similarly $\mathbf n_\Sigma \cdot e_{n+1} = (\mathbf n_{\Sigma'} \cdot e_r) \sin \theta$. \end{proof}
\section{Spectrum of the stability operator}
\subsection{Restricted spectra on compact domains}
In this section, we will study the spectrum of the operator~$L$. In light of Propositions \ref{fouriersum}~and~\ref{lkpropn},
it is natural to do so by decomposing the spectrum into separate parts corresponding to each $C^\infty(k, \Sigma)$. 

We will develop this spectral theory in the compact setting first. Let $D$ be a bounded, rotationally symmetric subdomain of~$\Sigma$.
\begin{defn}By the \emph{$k$-th restricted (Dirichlet) eigenvalue spectrum} of $L$ on $D$,  we mean the set of~$\mu \in \bbR$ for which there exists a nonzero function~$f \in C^\infty_0(k, D)$ such that $Lu + \mu u = 0$.
\end{defn}
\begin{rmk}
If $\Sigma \cap A = \emptyset$, then the restricted spectrum is just the spectrum of the operator~$L_k$, and the results of this section are standard consequences of the spectral theory of self-adjoint elliptic operators.

When $\Sigma$ does intersect $A$, we need to work harder. On one hand, the coefficients of~$L_k$ are unbounded near $A$, and the standard elliptic theory of~$L_k$ does not immediately give enough control on the behavior of the solutions near this boundary. On the other hand, functions in~$C^\infty(D')$ coming from $C^\infty(k, D)$ have restrictions on their behavior near $\partial \Sigma'$ (see Lemma~\ref{specialbehavior}), and this will help us finish the proof.\end{rmk}

We will work in a (weighted) Sobolev setting, and it will be convenient to define a bilinear form corresponding to~$L$. If $f, g \in C_c^\infty(D)$, we let
\begin{equation*}
B_L(f, g) = [-fLg]_D = \left[\nabla f \cdot \nabla g - \left(\abs{A}^2 + \frac 12\right)fg\right]_D\text.
\end{equation*}
In light of the second expression, $B_L$ extends to a bounded bilinear form on~$H^1_{0,\mathrm w}(D)$.

The restricted spectra satisfy all the usual properties we expect from spectra of second-order elliptic operators: 
\begin{thm}
\label{compactspectraltheorem}
Fix $k$, and let $D$ be a rotationally symmetric, bounded, open subset of the rotationally symmetric shrinker $\Sigma$.
\begin{enumerate}[(i)]
\item The $k$-th restricted eigenvalue spectrum of~$L$ consists of real eigenvalues $\mu_1 < \mu_2 \leq \mu_3 \leq \cdots$, where the corresponding eigenfunctions $f_1, f_2, \ldots$ form an orthonormal basis for~$L^2_{\mathrm w}(k, D)$.
\item  $\mu_1$ can be characterized by
\begin{equation}
\label{muinf}
\mu_1 = \inf_f B_L(f, f)
\end{equation}
where the infimum is taken over functions~$f \in H^1_{0,\mathrm w}(k, D)$ with $\norm{f}_{L^2_{\mathrm w}(D)} = 1$.
\item If some~$f \in H^1_{0, \mathrm w}(k, D)$ with $\norm f_{L^2_{\mathrm w}(D)} = 1$ satisfies $\mu_1 = B_L(f, f)$, then $Lf + \mu_1 f = 0$.
\item The lowest eigenfunction~$f_1$ can be written as $f_1(x) = u(x') \cos(k\theta)$, where $u > 0$ in the interior of~$D' \setminus \partial \Sigma'$.
\item If $k=0$, then the lowest eigenfunction~$f$ in $C^\infty_0(0, D)$ is in fact also the lowest eigenfunction for~$L$ acting on~$C^\infty_0(D)$ (i.e. in the unrestricted spectrum), and $f$ is in fact positive in the whole interior of $D$.
\end{enumerate}
\end{thm}
\begin{rmk}For unrestricted spectra, these results are standard; see e.g.\ Section~6.5 of~\cite{evans}, where the results are proved in a slightly simpler setting. The proof here uses the same ideas, with the next lemma being the only new ingredient. The lemma can be thought of as a weak version of the statement that $L$ acts on the spaces $C^\infty(k, D)$.\end{rmk}
\begin{lemma}
\label{borthlemma}Let $D$ be a rotationally symmetric, bounded, open subset of the rotationally symmetric shrinker $\Sigma$, and fix an integer $k \geq 0$.
\begin{enumerate}[(i)]
\item Suppose that $f \in H^1_{0, \mathrm w}(k, D)$ and $g \in H^1_{0, \mathrm w}(D)$, and let $g_k$ denote the $L^2_{\mathrm w}$-orthogonal projection of~$g$ onto the subspace $L^2_{\mathrm w}(k, D)$. Then $B_L(f, g) = B_L(f, g_k)$. 
\item Let $f \in H^1_{0,\mathrm w}(k,D)$ and let $\alpha \in \bbR$. Suppose that the equation
\begin{equation}
\label{bortheqn}
B_L(f, h) + \alpha [fh]_D = 0
\end{equation}
holds for all~$h \in H^1_{0,\mathrm w}(k, D)$. Then in fact $B_L(f, g) + \alpha [fg]_D = 0$ holds for any~$g \in H^1_{0, \mathrm w}(D)$.
\end{enumerate}
\end{lemma}
\begin{proof}$ $
\begin{enumerate}[(i)]
\item Suppose first that $f$ and $g$ are smooth. Proposition~\ref{lkpropn} gives $Lf \in C^\infty(k, D)$, and by orthogonality of the spaces $L^2_{\mathrm w}(k, D)$, we have
\begin{equation*}
B_L(f, g) = -[(Lf) g]_D = -[(Lf)g_k]_D = B_L(f, g_k)\text,
\end{equation*}
which is the desired identity.

In the general case, Lemma~\ref{h1decomp} implies that $g \mapsto g_k$ is continuous as a map $H^1_{0,\mathrm w}(D) \rightarrow H^1_{0,\mathrm w}(k, D)$, and the result follows from a continuity argument.
\item Decompose $g$ as $g = g_k + g'$. By part (i), $B_L(f, g) = B_L(f, g_k)$, and by simple orthogonality $\alpha [fg]_D = \alpha [fg_k]_D$. Thus 
\begin{equation*}
B_L(f, g) + \alpha[fg]_D = B_L(f, g_k) + \alpha[fg_k]_D\text,
\end{equation*}
which is zero by assumption.
\end{enumerate}
\end{proof}
\begin{proof}[Proof of Theorem \ref{compactspectraltheorem}]$ $
\begin{enumerate}[(i)]
\item  We can pick $\gamma > \sup_D \abs{A}^2 + \frac 12$ and define $B_\gamma(f, g) = B_L(f, g) + [fg]_D$. Then $B_\gamma$ defines an inner product on $H^1_{0,\mathrm w}(k,D)$, which satisfies $\norm{f}_{H^1_{\mathrm w}} \leq C B_\gamma(f, f)$ for some $C$.

We now define an operator $R:L^2_{\mathrm w}(k,D) \rightarrow H^1_{0,\mathrm w}(k,D)$ as follows. Let $h \in L^2_{\mathrm w}(k, D)$. By the Riesz representation theorem, applied to the Hilbert space $(H^1_{0,\mathrm w}(k,D), B_\gamma)$, we can find a unique $f \in H^1_{0, \mathrm w}(k, D)$ such that $B_\gamma(f, g) = [hg]_{D}$ for all $g \in H^1_{0,\mathrm w}(k,D)$. We then set $Rh = f$; it is easy to see that $R$ is bounded.

By the Rellich-Kondrachov theorem, $R$ is compact when thought of as a operator $L^2_{\mathrm w}(k, D) \rightarrow L^2_{\mathrm w}(k,D)$. It is also self-adjoint and positive definite, and it follows that it has a discrete spectrum of eigenvalues $\lambda_1 \geq \lambda_2 \geq \cdots > 0$, with $\lambda_i \rightarrow 0$, and with the corresponding eigenfunctions~$f_i$ forming an orthonormal basis of~$L^2_{\mathrm w}(k, D)$.

By the definition of $R$, the eigenfunctions $f_i$ satisfy 
\begin{equation}
\label{bgammaeqn}
 B_\gamma(f_i, g) = \lambda_i^{-1} [f_i g]_{D}
\end{equation}
for all $g \in H^1_{0,\mathrm w}(k, D)$. By Lemma \ref{borthlemma}(ii), the equation \eqref{bgammaeqn} in fact holds for all $g \in H^1_{0,\mathrm w}(D)$. In other words, $f_i$ is a weak solution of $(-L+\gamma)f_i = \lambda_i^{-1} f_i$. By standard elliptic regularity, it follows that the $f_i$ are smooth, and that the relation $(L+\gamma)f_i = \lambda_i^{-1} f_i$ holds in the classical sense.

Setting $\mu_i = \lambda_i^{-1} - \gamma$, this proves (i), except for the assertion that $\mu_1 < \mu_2$, which we will leave for later.
\item On one hand, $B_L(f_1, f_1) = \mu_1$, so the infimum is at most $\mu_1$. 

On the other hand, suppose $f \in H^1_{0,\mathrm w}(k, D)$, with $\norm{f}_{L^2_{\mathrm w}(D)} = 1$. We need to show that $B_L(f, f) \geq \mu_1$, or equivalently, that $B_\gamma(f, f) \geq \mu_1 + \gamma = \lambda_1^{-1}$. 

Since the eigenfunctions form an orthonormal basis for $L^2_{\mathrm w}(k, D)$, we can write $f = \sum_{i=1}^\infty a_i f_i$, with the convergence in $L^2_{\mathrm w}$. The coefficients $a_i$ satisfy $\sum_{i=1}^\infty a_i^2 = 1$. 

We also have the orthogonality relations
\begin{equation*}
B_\gamma(f_i, f_j) = [-f_j L f_i]_D = \begin{cases}\mu_i + \gamma &\text{if $k=l$} \\ 0 &\text{otherwise.}\end{cases}
\end{equation*}
By Bessel's inequality, applied to the Hilbert space $(H^1_{0,\mathrm w}(k, D), B_\gamma)$, it follows that $B_\gamma(f, f) \geq \sum_{i=1}^N \abs{B_\gamma(f, f_i)}^2 (\mu_i + \gamma)^{-1}$ for all $N$. Furthermore, $B_\gamma(f, f_i) = [-fLf_i]_D = (\mu_i + \gamma) [f f_i]_D = (\mu_i + \gamma)a_i$, so we get $B_\gamma(f, f) \geq \sum_{i=1}^N(\mu_i + \gamma)a_i^2 \geq (\mu_1 + \gamma) \sum_{i=1}^N a_i^2$. This holds for all~$N$, so we conclude that $B_\gamma(f, f) \geq \mu_i + \gamma$, which proves~(ii).
\item Let $g \in H^1_{0, \mathrm w}(k, D)$. I claim that 
\begin{equation}
\label{blfg}
B_L(f, g) = \mu_1 [fg]_D\text.
\end{equation}
By decomposing $g$ as $g = [fg]_D f + g'$, we see that it is sufficient to prove \eqref{blfg} for $g$ orthogonal to $f$ in $L^2_{\mathrm w}(D)$. For such $g$, we have $\frac{d}{ds}\norm{f+sg}_{L^2_{\mathrm w}}\rvert_{s=0} = 0$, and since $f$ minimizes $B_L(f, f)$, it follows that $\frac{d}{ds} B_L(f+sg,f+sg) = 0$, or in other words, $B_L(f, g) = 0$. This proves~\eqref{blfg} for $g \in H^1_{0,\mathrm w}(k, D)$.

By Lemma~\ref{borthlemma}(ii), the equation~\eqref{blfg} is in fact valid for all $g \in H^1_{0, \mathrm w}(D)$. In other words, $Lf + \mu_1 f = 0$ in the weak sense, and by elliptic regularity this holds in the classical sense as well. 
\item Suppose that $f_1 \in C^\infty_0(k, D)$ satisfies $Lf_1 + \mu_1 f_1 = 0$. Without loss of generality, assume that $\norm{f_1}_{D} = 1$. Write $f_1 = u(x') \cos(k\theta)$, for some function~$u$ on $D' \subset \Sigma'$. Consider the function $\tilde f_1 = \abs{u}\cos(k\theta)$. Since $B_L(\tilde f_1, \tilde f_1) = B_L(f_1, f_1) = \mu_1$ and $\norm{\tilde f_1}_{D} = 1$, it follows from part~(iii) that $L\tilde f_k + \mu_1 \tilde f_1 = 0$, and by Proposition~\ref{lkpropn}, we get $L_k \abs{w} + \mu_1 \abs{w} = 0$. 

In any compact region away from the boundary of~$D'$, $L_k$ is a uniformly elliptic operator with bounded coefficients. By applying the Harnack inequality (see e.g.\ Theorem~8.20 of~\cite{gilbargtrudinger}), we conclude that $\abs{u} > 0$ away from~$\partial D'$. In other words, $u$ does not vanish anywhere except possibly on~$\partial \Sigma'$. By Lemma~\ref{sigmaprimeprops}, $\Sigma' \setminus \partial \Sigma'$ is connected, and by changing the sign of~$u$ if necessary, we have shown that $u > 0$ in the interior of~$D'$.

This also shows that $\mu_1 < \mu_2$. Assume otherwise. Then there are two linearly independent eigenfunctions of eigenvalue $\mu_1$. By taking linear combinations of those, we can construct an eigenfunction of the form $f_1 = u \cos(k\theta)$ where $u$ changes sign. However, such an eigenfunction cannot exist by the preceding argument. This finishes the proof of part (i).
\item By standard spectral theory, $L^2_{\mathrm w}(D)$ has an orthonormal basis consisting of Dirichlet eigenfunctions of $L$. Furthermore, the lowest eigenfunction~$g$ is unique, and positive in the interior of~$D$. Since it is unique, it must be rotationally invariant, i.e.\ $g \in C^\infty(0, \Sigma)$. Thus, the eigenvalue corresponding to~$g$ is $\mu_1$, and by the uniqueness of the lowest eigenvalue (i.e\ the fact that $\mu_1 < \mu_2$), we conclude that $g = cf$ for some constant $c$. The claimed properties then follow.
\end{enumerate}
\end{proof}
\subsection{Non-compact domains}
When $\Sigma$ is non-compact, we will analyze the spectrum of~$L$ by taking a exhaustion of $\Sigma$ by compact (rotationally symmetric) domains. The actual choice of domains will not matter, but for definiteness, we will define
\begin{equation*}
D_R = \Sigma \cap B_R^{\bbR^{n+1}}(0)\quad(R\geq 0)\text,
\end{equation*}
where $B_R^{\bbR^{n+1}}(0)$ denotes the ball of radius $R$ centered at the origin in $\bbR^{n+1}$.

Let $\mu_1(k, D_R)$ denote the lowest eigenvalue in the $k$-th restricted Dirichlet eigenvalue spectrum of~$L$ on the domain~$D_R$. From the characterization in part~(ii) of Theorem~\ref{compactspectraltheorem}, we see that $\mu_1(k, D_S) \leq \mu_1(k, D_R)$ whenever $S \geq R$. So we can make the following definition:
\begin{defn}Given a smooth immersed hypersurface~$\Sigma$, we define the \emph{bottom of the $k$-th restricted spectrum} of~$L$, denoted $\mu_1(k, \Sigma)$, by
\begin{equation*}
\mu_1(k, \Sigma) = \lim_{R \rightarrow \infty} \mu_1(k, D_R)\text.
\end{equation*}
We allow the value $\mu_1(k,\Sigma) = - \infty$.
\end{defn}
We can also take a limit of the corresponding eigenfunctions:
\begin{lemma}
\label{vsublimitlemma}
Let $\Sigma$ be a rotationally symmetric shrinker.
Fix a compact subset~$K'$ of~$\Sigma' \setminus \partial \Sigma'$, a point~$x' \in K'$, and an integer $k \geq 0$. For each $R > 0$, let $f_R$ denote (some multiple of) the eigenfunction of~$L$ corresponding to the eigenvalue $\mu_1(k, D_R)$. We write it as $f_R = v_R \cos(k \theta)$, where $v_R$ is a function on~$D_R'$ which is positive in the interior of $D_R'$.. (This is possible by Theorem~\ref{compactspectraltheorem}(iv).) We scale $f_R$ so that $v_R(x') = 1$.

Assume that $\mu_1(k, \Sigma) > -\infty$. Then the following hold:
\begin{enumerate}[(i)]
\item The bounds $0 < c \leq v_R \leq C$ hold inside $K'$, where the constants $c$ and $C$ do not depend on $R$ (but do depend on~$K'$).
\item We can find a sequence $R_i \rightarrow \infty$ such that $v_{R_i}\rvert_{K'}$ converges in $C^1(K')$.
\item Let $w_R = \log v_R$. Then $w_{R_i}$ also converges in~$C^1(K')$ to some $w_\infty \in C^1(K')$.
\end{enumerate}
\end{lemma}
\begin{proof}$ $
\begin{enumerate}[(i)]
\item $v_R$ satisfies the elliptic equation $L_k v_R + \mu_1(k, D_R) v_R = 0$. Since $\mu_1(k, \Sigma) > -\infty$, this equation has bounded coefficients, with the bound not depending on $R$. Since we normalized the solution to have $v_R(x') = 1$, a standard Harnack inequality (see e.g.\ Theorem~8.20 of~\cite{gilbargtrudinger}) now gives the estimate $0 < c \leq v_R \leq C$ in the compact domain $K'$. 
\item Let $E'$ be a slightly larger compact domain, such that $K'$ is contained in the interior of $E'$. Applying part (i) to the domain $E'$, we have $\abs{v_R} \leq C$ in $E'$, for some different constant $C$.

We can now apply a standard interior $C^{1, \alpha}$-estimate for elliptic equations, Theorem~8.32 of~\cite{gilbargtrudinger}, in the domain $E'$ to show that that the $C^{1,\alpha}$-norm of $v_R$ in $K'$ is bounded. 

Finally, the Arzel\`a-Ascoli theorem then gives a subsequence~$v_{R_i}$ which converges in the $C^1$-norm on $K'$.
\item This is an immediate consequence of parts (i)~and~(ii). Note that the lower bound $0 < c \leq v_R$ is needed to control the denominator of $\abs{\nabla \log u_R} = \abs{\nabla v_R}/\abs{w_R}$.
\end{enumerate}
\end{proof}
Before we state the main result of this section, we need another lemma, containing a calculation that will be helpful later:
\begin{lemma}$ $
\label{helpfulibp}
Let $\Sigma$ be a rotationally symmetric immersed hypersurface in $\bbR^{n+1}$.
\begin{enumerate}[(i)]
\item Suppose that $f, \phi \in C^\infty(\Sigma)$ and that $\phi$ has compact support. Then
\begin{equation*}
\left[\abs{\nabla(\phi f)}^2\right]_\Sigma = \left[\abs{\nabla \phi}^2 f^2 - \phi^2 f \mathcal L f\right]_\Sigma\text.
\end{equation*}
\item Suppose that $u, \phi \in C^\infty(\Sigma')$ (up to the boundary), and that $\phi$ has compact support. Then
\begin{equation*}
\left[\abs{\nabla(\phi u)}^2\right]'_{\Sigma'} = \left[\abs{\nabla \phi}^2 u^2 - \phi^2 u \mathcal L' u\right]'_{\Sigma'}\text.
\end{equation*}
\end{enumerate}
\end{lemma}
\begin{proof}$ $
\begin{enumerate}[(i)]
\item Using the Leibniz rule and expanding, we get
\begin{equation}
\label{cutoffibpexp}
\left[\abs{\nabla(\phi f)}^2\right]_\Sigma
= \left[\abs{\nabla \phi}^2 f^2 + 2\phi f \nabla \phi \cdot \nabla f + \phi^2 \abs{\nabla f}^2\right]_\Sigma\text.
\end{equation}
By integrating by parts (see Proposition~\ref{intbp}), we have
\begin{align*}
-\left[\phi^2 f \mathcal L f\right]_\Sigma
&= \left[\nabla(\phi^2 f) \cdot \nabla f\right]_\Sigma \\
&= \left[2\phi f \nabla \phi \cdot \nabla f + \phi^2 \abs{\nabla f}^2\right]_\Sigma
\end{align*}
Substituting this into~\eqref{cutoffibpexp} gives the result. 
\item The proof follows the same steps as in case (i). First, we expand to get
\begin{equation*} 
\label{cutoffibpexpprime}
\left[\abs{\nabla(\phi u)}^2\right]'_{\Sigma'}
= \left[\abs{\nabla \phi}^2 u^2 + 2\phi u \nabla \phi \cdot \nabla u + \phi^2 \abs{\nabla u}^2\right]'_{\Sigma'}\text.
\end{equation*}
By Proposition~\ref{intbpprime}, we have
\begin{align*}
-\left[\phi^2 u \mathcal L' u\right]'_{\Sigma'}
&= \left[\nabla(\phi^2 u) \cdot \nabla u\right]'_{\Sigma'} \\
&= \left[2\phi u \nabla \phi \cdot \nabla u + \phi^2 \abs{\nabla u}^2\right]'_{\Sigma'}
\end{align*}
and substitution into~\eqref{cutoffibpexpprime} gives the result.
\end{enumerate}
\end{proof}
We now come to the main result of this section, which allows us to use global eigenfunctions to gain spectral information, even in the non-compact case.
\begin{thm}
\label{bottomofspectrumtest}
Let $\Sigma$ be a rotationally symmetric shrinker with polynomial volume growth, and suppose that $f \in C^\infty(k, \Sigma)$ satisfies $Lf + \lambda f = 0$. Assume also that $f$ has polynomial growth, i.e.\ that $\abs{f} \leq p(\abs{x})$ for some polynomial~$p$.
\begin{enumerate}[(i)]
\item Then $\mu_1(k, \Sigma) \leq \lambda$.
\item We can write $f(x) = u(x')\cos(k\theta)$. If, in addition, $u$ changes sign in $\Sigma' \setminus \partial \Sigma'$, then $\mu_1(k, \Sigma) < \lambda$.
\end{enumerate}
\end{thm}
\begin{proof}$ $
\begin{enumerate}[(i)]
\item Let $\phi_R \in C^\infty(\Sigma)$ be a cut-off function which depends only on~$\abs{x}$, takes the value $0$ outside of $D_R$, takes the value~1 inside~$D_{R/2}$, and which is approximately linear in~$\abs{x}$  in between. We will use $\phi_R f$ as a test function in the definition of~$\mu_1(k, D_R)$.

 Using Lemma~\ref{helpfulibp}(i), we have
\begin{equation*}
\left[\abs{\nabla(\phi_R f)}^2\right]_\Sigma
= \left[\abs{\nabla \phi_R}^2 f^2 - \phi_R^2 f \mathcal L f\right]_\Sigma\text.
\end{equation*}
This means that
\begin{align*}
[-\phi_R f L( \phi_R f)]_\Sigma &=
\left[\abs{\nabla(\phi_R f)}^2 - \left(\abs A^2 + \frac 12\right) \phi_R^2 f^2 \right]_\Sigma \\
&= \left[\abs{\nabla \phi_R}^2 f^2 - \phi_R^2 f \mathcal L f - \left(\abs{A}^2 + \frac 12 \right)\phi_R^2 f^2\right]_\Sigma \\
&= \left[\abs{\nabla \phi_R}^2 f^2 - \phi_R^2 f L f\right]_\Sigma \\
&= \left[\abs{\nabla \phi_R}^2 f^2\right]_\Sigma + \lambda [\phi_R^2 f^2]_\Sigma\text.
\end{align*}
We now take the limit $R \rightarrow \infty$. Note that the term $[\abs{\nabla \phi_R}^2 f^2]_\Sigma$ is a weighted integral with weight~$e^{-\abs{x}^2/4}$ and integrand supported only in~$\Sigma \setminus D_R$. From the polynomial growth of~$f$, the fact that $\abs{\nabla \phi_R} \leq C/R$ for some constant~$C$, and the fact that $\Sigma$ has polynomial volume growth, it follows that this term vanishes in the limit. In other words,
\begin{equation*}
[-\phi_R fL(\phi_R f)]_\Sigma - \lambda [\phi_R^2 f^2]_\Sigma \rightarrow 0
\end{equation*}
as $R \rightarrow \infty$. Since $[\phi_R^2 f^2]_\Sigma$ is increasing in~$R$, we can divide through to obtain
\begin{equation*}
\lim_{R \rightarrow \infty} \frac{[-\phi_R fL(\phi_R f)]_\Sigma}{[\phi_R^2 f^2]_\Sigma} = \lambda\text.
\end{equation*}
For any fixed large~$R$, $\phi_R f$ is a member of $C^\infty_0(k, D_S)$ for $S$ sufficiently large. Then the definition of $\mu_1(k,D_S)$ gives $\mu_1(k, D_S) \leq \frac{[-\phi_R fL(\phi_R f)]_\Sigma}{\norm{\phi_R f}_\Sigma^2}$, and it follows that $\mu_1(k, \Sigma) \leq \lambda$.
\item Assume, for contradiction, that $\lambda = \mu_1(k, \Sigma)$. Let $g_R$ denote the lowest eigenfunction for $L$ in the $k$-th restricted spectrum, i.e.\ corresponding to the eigenvalue~$\mu_1(k, D_R)$. By Theorem~\ref{compactspectraltheorem}(iv), we can write $g_R(x) = v_R(x') \cos(k\theta)$, where $v_R > 0$ in the interior of $D_R' \setminus \{r=0\}$. We can then define $w_R:D_R' \setminus \{r=0\} \rightarrow \bbR$ by $w_R(x') = \log v_R(x')$. A simple calculation shows that 
\begin{equation}
\label{lprimewr}
\mathcal L' w_R = -\mu_1(k, D_R) - c - \abs{\nabla w_R}^2\text,
\end{equation}
where the function~$c$ is the degree-zero coefficient of~$L_k$: $c = \abs{A}^2 + \abs{e_r^\perp}^2/r^2 + \frac 12 - k^2/r^2$.

We split into two cases.
\begin{description}
\item[Case 1: $k \geq 1$.]
Assume first that $k \geq 1$. Note that $L_k u + \lambda u = 0$. 

We will use two different cut-off functions, one to change $u$ to have compact support, and one to move the support away from the boundary $\partial \Sigma'$:
\begin{itemize}
\item $\phi_R \in C^\infty(\Sigma')$ has compact support inside $D_R'$, takes the value 1 inside $D_{R/2}'$, and is approximately linear in $\abs{x'}$ in between.
\item $\psi_\rho \in C^\infty(\Sigma')$ takes the value $0$ in the region $\Sigma' \cap \{r < \rho\}$, takes the value $1$ in the region $\Sigma' \cap \{r > 2\rho\}$, and is approximately linear in $r$ in between.
\end{itemize}
We can choose these functions so that $\abs{\nabla \phi_R} \leq C/R$, and $\abs{\nabla \psi_\rho} \leq C/\rho$, for some universal constant $C$.

Now define $u_\rho^R = \phi_R \psi_\rho u$. We will compute the integral
\begin{equation}
\label{irhordefn}
I^R_\rho = \left[\abs{\nabla u_\rho^R - u_\rho^R \nabla w_R}^2\right]'_{D_R'}\text.
\end{equation}
By expanding, we get
\begin{align*}
I^R_\rho &= \left[ \abs{\nabla u_\rho^R}^2\right]'_{D_R'} - \left[2u_\rho^R \nabla u_\rho^R \cdot \nabla w_R\right]'_{D_R'} + \left[(u_\rho^R)^2 \abs{\nabla w_R}^2\right]'_{D_R'} \\
&\equalscolon A + B + C
\end{align*}
where $A$, $B$, and $C$ are defined by this last equality.

By Lemma \ref{helpfulibp}(ii), we have
\begin{align*}
A &= \left[\abs{\nabla u_\rho^R}^2\right]'_{D_R'} \\
&= \left[\abs{\nabla (\psi_\rho \phi_R)}^2 u^2 - \psi_\rho^2 \phi_R^2 u \mathcal L' u\right]'_{D_R'}\text.
\end{align*}

By using the divergence theorem (i.e.\ Proposition~\ref{intbpprime}), we have
\begin{align*}
B &= \left[-\nabla(u_\rho^R)^2 \cdot \nabla w_R\right]'_{D_R'} \\
&= \left[(u_\rho^R)^2 \mathcal L' w_R \right]'_{D_R'}\text.
\end{align*}
Combining this with \eqref{lprimewr}, we get
\begin{equation*}
B + C = \left[ (u_\rho^R)^2 \left(-\mu_1(k, D_R) - c\right)\right]'_{D_R'}
\end{equation*}
and finally
\begin{align*}
I_\rho^R &= A + B + C  \\
&= \left[\abs{\nabla (\psi_\rho \phi_R)}^2 u^2 - \psi_\rho^2 \phi_R^2 u\left(L_k u + \mu_1(k, D_R) u\right)\right]'_{D_R'} \\
&= \left[\left(\abs{\nabla(\psi_\rho \phi_R)}^2 + \psi_\rho^2 \phi_R^2 (\lambda - \mu_1(k, D_R)) \right) u^2 \right]'_{D_R'}\text.
\end{align*}
Assume furthermore that $2\rho < R/2$, so that $\nabla \psi_\rho \cdot \nabla \phi_R \equiv 0$. We get
\begin{equation}
\label{beforerholimit}
I_\rho^R = \left[\left(\abs{\nabla \psi_\rho}^2 + \abs{\nabla \phi_R}^2 + \psi_\rho^2 \phi_R^2(\lambda - \mu_1(k, D_R))\right)u^2\right]'_{D'_R}\text.
\end{equation}
We will first take the limit $\rho \rightarrow 0^+$. First of all, I claim that
\begin{equation}
\label{rholimit}
\lim_{\rho \rightarrow 0^+} \left[ \abs{\nabla \psi_\rho}^2 u^2\right]'_{D_R'} = 0\text.
\end{equation}
To see this, note that for $R$ fixed, we have
\begin{equation*}
[1]_{\{x' \in D_R' \mid r < 2\rho\}}' = O(\rho^2)
\end{equation*}
as $\rho \rightarrow 0^+$. Since $\nabla \psi_\rho$ is only supported in $\{r < 2 \rho\}$, and $\abs{\nabla \psi_\rho} \leq C/\rho$, we get 
\begin{equation*}
\left[\abs{\nabla\psi_\rho}^2\right]'_{D_R'} \leq C'
\end{equation*}
for some constant~$C'$ that depends on~$R$ but not on~$\rho$. Furthermore, $u\rightarrow 0$ as $r \rightarrow 0$ by Lemma~\ref{specialbehavior}(ii), and since $D_R'$ is compact, this convergence is uniform inside $D_R'$. Now \eqref{rholimit} follows.

Combining \eqref{beforerholimit} and \eqref{rholimit}, we get
\begin{equation*}
\lim_{\rho\rightarrow 0^+} I_\rho^R =
\lim_{\rho \rightarrow 0^+} \left[\left(\abs{\nabla \phi_R}^2 + \psi_\rho^2 \phi_R^2(\lambda - \mu_1(k, D_R))\right)u^2\right]'_{D'_R}\text,
\end{equation*}
and the dominated convergence theorem gives 
\begin{equation}
\label{irhorlimit}
\lim_{\rho \rightarrow 0^+} I_\rho^R 
= \left[\left(\abs{\nabla \phi_R}^2 + \phi_R^2 (\lambda - \mu_1(k, D_R))\right)u^2\right]'_{D_R'}\text.
\end{equation}

On the other hand, we can let $u^R = \phi_R u$, and define
\begin{equation*}
I^R =  \left[\abs{\nabla u^R - u^R \nabla w_R}^2\right]'_{D_R'}\text.
\end{equation*}
By taking the limit $\rho \rightarrow 0^+$ directly in the definition \eqref{irhordefn} of $I_\rho^R$, Fatou's lemma gives $I^R \leq \liminf_{\rho \rightarrow 0^+} I_\rho^R$, and combining with \eqref{irhorlimit} gives
\begin{equation}
\label{irinequality}
I^R \leq \left[\left(\abs{\nabla \phi_R}^2 + \phi_R^2 (\lambda - \mu_1(k, D_R))\right)u^2
\right]'_{D_R'}\text.
\end{equation}
Now we take the limit $R \rightarrow \infty$. This is simpler: $\mu_1(k, D_R) \rightarrow \lambda$ by assumption, and $\left[\abs{\nabla \phi_R}^2 u^2\right]'_{D_R'} \rightarrow 0$ by the bounds on $\abs{\nabla \phi_R}$ and because $u$ has polynomial growth. We thus get
$\limsup_{R \rightarrow \infty} I^R \leq 0$.

On the other hand, we can take the limit directly in the definition of~$I^R$. Let $K'$ be a fixed connected compact subset of $\Sigma' \setminus \partial \Sigma'$. By Lemma~\ref{vsublimitlemma}(iii), we can find radii $R_i \rightarrow \infty$ and a function~$w_\infty$ such that $w_{R_i} \rightarrow w_\infty$ in $C^1(K')$. Fatou's lemma then gives
\begin{equation*}
\left[\abs{\nabla u - u \nabla w_\infty}^2\right]'_{K'} 
\leq \liminf_{i \rightarrow \infty} \left[\abs{\nabla u^{R_i} - u^{R_i} \nabla w_{R_i}}^2\right]'_{K'}\text.
\end{equation*}
For $R$ sufficiently large, $K' \subseteq D_R$, and so
\begin{equation*}
\left[\abs{\nabla u - u \nabla w_\infty}^2 \right]'_{K'}
\leq \liminf_{R \rightarrow \infty} I^R\text.
\end{equation*}
Combining with $\limsup_{R\rightarrow\infty} I^R \leq 0$, we conclude that
\begin{equation*}
\left[\abs{\nabla u - u \nabla w_\infty}^2\right]'_{K'}\text.
\end{equation*}
This means that $\nabla u = u \nabla w_\infty$ everywhere in~$K'$. Since $K'$ is connected, it follows that $u = D e^{w_\infty}$ for some constance $D$, and in particular, $u$ does not change sign on $K'$. Since $K'$ is an arbitrary connected compact subset of the interior of $\Sigma'$, and $\Sigma'$ is connected, this contradicts the assumption that $u$ changes sign, and finishes the proof.
 \item[Case 2: $k=0$.] In this case, we do not have the result that $u \rightarrow 0$ near the boundary $\partial \Sigma'$. On the other hand, Theorem~\ref{compactspectraltheorem}(v) ensures that $w_R$ can be defined on the boundary~$\partial \Sigma'$, which removes the need for the cutoff function~$\psi_\rho$ altogether, and we end up with a simpler proof.
 
 Let $u^R = \phi_R u$, and consider the integral
 \begin{equation*}
I^R \colonequals \left[\abs{\nabla u^R - u^R \nabla w_R}^2\right]'_{D_R'}\text.
\end{equation*}
Expanding this, we get
 \begin{align*}
  I^R &= \left[\abs{\nabla u^R}^2\right]'_{D_R'} - \left[2 u^R \nabla u^R \cdot \nabla w_R\right]'_{D_R'} + \left[(u^R)^2 \abs{\nabla w_R}^2\right]'_{D_R'} \\
  &\equalscolon A + B + C\text,
 \end{align*}
 where the last equality defines the terms $A$, $B$, and $C$.
 
 We use Lemma \ref{helpfulibp}(ii) for term $A$:
 \begin{equation*}
  A = \left[\abs{\nabla \phi_R}^2 u^2 - \phi_R^2 u \mathcal L' u\right]'_{D_R'}\text.
 \end{equation*}
 Because $k=0$, Theorem~\ref{compactspectraltheorem}(v) guarantees that $w_R$ is smooth up to the boundary $\partial \Sigma'$. We can apply Proposition~\ref{intbpprime}, to get
 \begin{align*}
 B &= - \left[\nabla (u^R)^2 \cdot \nabla w_R\right]'_{D_R'}\\
 &= \left[ (u^R)^2 \mathcal L' w_R \right]'_{D_R'}\text.
 \end{align*}
 Equation~\eqref{lprimewr} then gives
 \begin{equation*}
 B + C = \left[ (u^R)^2 \left(-\mu_1(k, D_R) - c\right)\right]'_{\Sigma'}
 \end{equation*}
 and finally
 \begin{align*}
 I^R &= A + B + C \\
 &=
 \left[\abs{\nabla \phi_R}^2 u^2 - \phi_R^2 u\left(L_k u + \mu_1(k, D_R) u\right)\right]'_{D_R'} \\
 &= \left[\left(\abs{\nabla \phi_R}^2 + \phi_R^2 (\lambda - \mu_1(k, D_R))\right)u^2 \right]'_{D_R'} \text.
 \end{align*}
 This is now the same expression as \eqref{irinequality} (except with equality), and we finish the proof in exactly the same way as in Case 1.
\end{description}
\end{enumerate}
\end{proof}
\section{$F$-unstable variations}
\subsection{Almost orthogonality of spacetime translations}
The proof of our main result, Theorem~\ref{maintheorem}, will depend on the spectral theory developed in the previous section, combined with an analysis of the two classes of eigenfunctions of $L$ we already know, namely $H$ and $y \cdot \mathbf n$ from Proposition~\ref{knowneigenfunctions}.

Throughout this section, $\Sigma$ will be an immersed shrinker of polynomial volume growth, and like in the previous section, $D_R$ will be used to denote the subdomain $\Sigma \cap B_R^{\bbR^{n+1}}(0)$.

First of all, we show that the eigenfunctions from Proposition~\ref{knowneigenfunctions} are almost orthogonal:
\begin{lemma}\label{timespaceorthogonal}Suppose that the rotationally symmetric shrinker~$\Sigma$, with polynomial volume growth, satisfies $\mu_1(0, \Sigma) > -\infty$. Let $y \in \bbR^{n+1}$ be a constant vector. Then
\begin{equation*}
[H(y \cdot \mathbf n)]_{D_R} \rightarrow 0
\end{equation*}
as $R \rightarrow \infty$.
\end{lemma}
\begin{proof}
Let $\phi$ be a cutoff function which takes the values $1$ inside~$D_R$ and $0$ outside~$D_{R+1}$, and which is approximately linear in~$\abs{x}$ in between. We can choose $\phi$ such that $\abs{\nabla \phi} \leq C$ for some universal constant~$C$.

Using Proposition~\ref{knowneigenfunctions}, we have
\begin{align*}
\frac 12 \left[\phi H y \cdot \mathbf n\right]_\Sigma 
&= \left[ \phi \left(H L (y \cdot \mathbf n) - LH y \cdot \mathbf n\right)\right]_\Sigma \\
&= \left[\phi \left( H \mathcal L (y \cdot \mathbf n) - \mathcal L H y \cdot \mathbf n\right) \right]_\Sigma \\
&= [\nabla(\phi y \cdot \mathbf n) \cdot \nabla H - \nabla(\phi H) \cdot \nabla (y \cdot \mathbf n)]_\Sigma \\
&= \left[\nabla \phi \cdot \left(\nabla H y \cdot \mathbf n - H\nabla (y \cdot \mathbf n)\right)\right]_\Sigma\text.
\end{align*}
By the shrinker equation, we have $H = \frac 12 x \cdot \mathbf n$. Therefore $\nabla_X H = \frac 12 A(x^T, X)$, and so $\abs{\nabla H} \leq \abs{A}\abs{x}$. Similarly, $\abs{\nabla(y \cdot \mathbf n)} \leq \abs{A}{\abs y}$. We get
\begin{equation*}
\frac 12 \abs{ \left[ \phi H y \cdot \mathbf n\right]_\Sigma} \leq 2\left[ \abs{\nabla \phi} 
\abs{A} \abs{x} \abs{y} \right]_\Sigma\text.
\end{equation*}
Applying Cauchy-Schwarz, we get
\begin{equation*}
\abs{\left[\phi H y \cdot \mathbf n \right]_\Sigma}  \leq C \abs{y} \left[\abs{A}^2\abs{x}^2\right]_\Sigma \left[1\right]_{D_{R+1} \setminus D_R}\text.
\end{equation*}
Since $\mu_1(0, \Sigma) > -\infty$, we know that $\left[\abs{A}^2\abs{x}^2\right]_{\Sigma} < \infty$ from \cite{cm} (see Step~1 in the proof of Theorem~9.36, and in particular formula (9.39)). By polynomial volume growth, we have that
$\left[1\right]_{D_{R+1} \setminus D_R} \leq \vol(D_{R+1}) e^{-R^2/4} \rightarrow 0$, and it follows that
\begin{equation*}
\left[\phi H y \cdot \mathbf n\right]_\Sigma \rightarrow 0\text.
\end{equation*}
On the other hand, we have
\begin{equation*}
\abs{\left[ \phi H y \cdot \mathbf n \right]_{\Sigma \setminus D_R}}
\leq \left[\abs{x}\abs{y}\right]_{D_{R+1} \setminus D_R} \leq \vol(D_{R+1}) (R+1) \abs{y} e^{-R^2/4}
\end{equation*}
which tends to zero by polynomial volume growth. Since $[H(y \cdot \mathbf n)]_{D_R} = \left[\phi H y \cdot n\right]_\Sigma - \left[\phi H y \cdot \mathbf n\right]_{\Sigma \setminus D_R}$, the result follows.
\end{proof}
In fact, we need a slightly stronger result, showing that Lemma~\ref{timespaceorthogonal} holds as a uniform limit in a certain sense:
\begin{lemma}
\label{timespaceuniform}
Suppose that the rotationally symmetric shrinker~$\Sigma$, with polynomial volume growth, satisfies $\mu_1(0, \Sigma) > -\infty$. Let $\in \bbR^{n+1}$ be a constant vector and let $\epsilon > 0$ be arbitrary. Then there exists an $R_0$, depending on $\epsilon$ and $\Sigma$ but not on $y$, such that 
\begin{equation}
\label{timespaceuniformresult}
\left[H(y\cdot \mathbf n)\right]_{D_R} \leq \epsilon \left[H^2\right]^{1/2}_{D_R} \left[(y \cdot \mathbf n)^2\right]^{1/2}_{D_R}
\end{equation}
holds whenever $R \geq R_0$.
\end{lemma}
\begin{proof}
The main idea of the proof is a simple compactness argument, but the possibility of $y \cdot \mathbf n \equiv 0$ gives rise to some technical complications.

If $[H^2]_{\Sigma} = 0$, then $H \equiv 0$ and so $[H(y \cdot \mathbf n)]_{D_R} = 0$ for all~$R$, and so the result is trivially true. Thus, we can assume that $[H^2]_{\Sigma} \neq 0$, and we let $R_1$ be such that $[H^2]_{D_{R_1}} > 0$.

For any~$R$, we let $A_R$ be the set of all $y \in \bbR^{n+1}$ such that $[(y \cdot \mathbf n)^2]_{D_R} = 0$. Since this condition is equivalent to $y \cdot \mathbf n$ being zero everywhere in $D_R$, $A_R$ is a linear subspace of $\bbR^{n+1}$. Since $R \mapsto A_R$ is monotonic, we can let $A = \lim_{R \rightarrow \infty} A_R$. Because $\bbR^{n+1}$ is finite-dimensional, there is some $R_2$ such that $A_R = A$ for all $R \geq R_2$. We then let $B$ be the orthogonal complement of $A$ in $\bbR^{n+1}$.

Consider the family of functions
\begin{equation*}
f_R(y) = \frac{\left[H(y\cdot \mathbf n)\right]_{D_R}}{\left[H^2\right]^{1/2}_{D_R}\left[(y \cdot \mathbf n)^2\right]^{1/2}_{D_R}}\text,
\end{equation*}
which we define for $y \in B \setminus \{0\}$ and $R \geq \max\{R_1, R_2\}$. Since the denominator is positive and increasing in $R$, Lemma~\ref{timespaceorthogonal} shows that $f_R(y) \rightarrow 0$ as $R \rightarrow \infty$, for each fixed $y \in B \setminus \{0\}$. Furthermore, note that $f$ is homogeneous of degree $0$ in $y$. From the fact that each $f_R$ is continuous in $y$, and the fact that the unit sphere in $B$ is compact, it then follows that $f_R \rightarrow 0$ as $R \rightarrow \infty$, with the convergence being uniform in $y$. This means that we can choose an $R_0 \geq \max\{R_1, R_2\}$ such that
\begin{equation}
\label{fruniformlimit}
f_R(y) \leq \epsilon
\end{equation}
for all $y \in B$.

Suppose now that $R \geq R_0$, and that $y \in \bbR^{n+1}$ is arbitrary. We decompose $y$ as $y = y_1 + y_2$ with $y_1 \in A$ and $y_2 \in B$. Since $y_1 \cdot \mathbf n \equiv 0$, we get
\begin{gather*}
\left[H(y \cdot \mathbf n)\right]_{D_R} = \left[H(y_2 \cdot \mathbf n)\right]_{D_R}\text{, and} \\
\left[(y\cdot \mathbf n)^2\right]_{D_R} = \left[(y_2 \cdot \mathbf n)^2\right]_{D_R}\text.
\end{gather*}
Thus \eqref{timespaceuniformresult} reduces to
\begin{equation*}
\left[H(y_2 \cdot \mathbf n)\right]_{D_R} \leq \epsilon \left[H^2\right]^{1/2}_{D_R} \left[(y_2 \cdot \mathbf n)^2\right]^{1/2}_{D_R}\text.
\end{equation*}
If $y_2 = 0$, then this is trivial; otherwise it follows from~\eqref{fruniformlimit}.
\end{proof}
\subsection{Bounds for the second variation}
To prove Theorem~\ref{maintheorem}, we will construct some unstable variations. To show that they are actually $F$-unstable, the strategy is to split the terms of the second variation formula (Proposition \ref{secondvariationformula}) up into different parts, corresponding to the different restricted spectra, and deal with each separately.

For one step, we will need the following elementary ``almost Bessel inequality'':
\begin{lemma}
\label{almostbessel}
Suppose that $V$ is a real vector space equipped with an inner product $\inner \cdot\cdot$ and corresponding norm $\norm\cdot$. Let $a, b, v \in V$, with $a$ and $b$ nonzero and $\abs{\inner ab} \leq \epsilon \norm a \norm b$, for some small $\epsilon$. Then
\begin{equation*}
\frac{\inner{v}{a}^2}{\norm{a}^2} + \frac{\inner{v}{b}^2}{\norm{b}^2} \leq \frac{1}{1 -\epsilon}\norm v^2\text.
\end{equation*}
\end{lemma}
\begin{proof}
Define $\hat a = a/\norm a$ and $\hat b = b / \norm b$. Then $\abs{\inner{\hat a}{\hat b}}\leq \epsilon$. 

We have
\begin{align*}
0 &\leq \norm{v - \inner v{\hat a} -  \hat a\inner v{\hat b}b}^2 \\
&= \norm v^2 - \inner v{\hat a}^2 - \inner v{\hat b}^2 + 2\inner{\hat a}{\hat b} \inner v{\hat a} \inner v{\hat b} \\
&\leq \norm v^2 - \inner v{\hat a}^2 - \inner v{\hat b}^2 +  + 2\epsilon \abs{ \inner v{\hat a} \inner v{\hat b}} \\
&\leq \norm v^2 - \inner v{\hat a}^2 - \inner v{\hat b}^2 + \epsilon(\inner v{\hat a}^2 + \inner v{\hat b}^2)\text.
\end{align*}
By rearranging, we get
\begin{equation*}
(1-\epsilon)\left(\inner v{\hat a}^2 + \inner v{\hat b}^2\right) \leq \norm v^2\text,
\end{equation*}
and the desired result follows.
\end{proof}
\begin{lemma}
\label{variationpartialbounds}
Let $f \in C^\infty_0(D_R)$, let $y \in \bbR^{n+1}$ be a fixed vector, and let $h \in \bbR$. Then
\begin{enumerate}[(i)]
\item
\begin{equation*}
\left[-\frac 12 f^2 + f y \cdot \mathbf n - \frac 12 (y \cdot \mathbf n)^2\right]_{D_R} \leq 0\text.
\end{equation*}
\item 
\begin{equation*}
\left[-\frac 32 f^2 + f y \cdot \mathbf n - \frac 12 (y \cdot \mathbf n)^2 + 2fhH - h^2 H^2\right]_{D_R} \leq 0\text.
\end{equation*}
\item If, in addition, we have
\begin{equation*}
\abs{[H(y \cdot \mathbf n)]_{D_R}}
\leq \epsilon [H]^{1/2}_{D_R} [(y \cdot \mathbf n)^2]^{1/2}_{D_R}
 \end{equation*}
 for some $\epsilon$ with $0 < \epsilon < 1$, then
\begin{equation*}
\left[-\frac{1}{1-\epsilon} f^2 + f y \cdot \mathbf n - \frac 12 (y \cdot \mathbf n)^2 + 2fhH - h^2 H^2\right]_{D_R} \leq 0\text.
\end{equation*}
\end{enumerate}
\end{lemma}
\begin{proof}$ $
\begin{enumerate}[(i)]
\item This is just an expanded version of the inequality
\begin{equation*}
- \frac 12 \left[(f - y\cdot \mathbf  n)^2\right]_{D_R} \leq 0\text.
\end{equation*}
\item This is obtained by taking the sum of (i) with the similar inequality
\begin{equation*}
-\left[(f - hH)^2\right]_{D_R}\leq 0 \text.
\end{equation*}
\item If either $hH$ or $y \cdot \mathbf n$ is identically zero on $D_R$, we can easily prove this using the same technique as in parts (i) and (ii). So we may assume that they are not.

We note that $0 \leq ([h^2 H^2]_{D_R} - [f hH]_{D_R})^2$, from which it follows that
\begin{equation*}
\left[2fhH - h^2H^2\right]_{D_R} \leq \frac{[f hH]_{D_R}^2}{[h^2H^2]_{D_R}}\text.
\end{equation*}
Similarly, we have
\begin{equation*}
\left[fy\cdot\mathbf n - \frac 12(y \cdot \mathbf n)^2 \right]_{D_R} \leq \frac 12\frac {[f (y\cdot\mathbf n)]^2_{D_R}}{[(y\cdot \mathbf n)^2]_{D_R}}\text.
\end{equation*}
Adding these together, we get
\begin{equation}
\label{almostorthineq1}
\left[f y \cdot \mathbf n - \frac 12 (y \cdot \mathbf n)^2 + 2fhH - h^2 H^2\right]_{D_R}
\leq \frac{[fhH]_{D_R}^2}{[h^2H^2]_{D_R}} + \frac 12\frac{[f(y\cdot\mathbf n)]^2_{D_R}}{[(y\cdot \mathbf n)^2]_{D_R}} \\
\end{equation}
Applying Lemma~\ref{almostbessel} with $V = L^2_{\mathrm w}(D_R)$, $\inner{f}{g} = [fg]_{D_R}$, $v=f$, $a = hH$, and $b = y\cdot\mathbf n$, we get
\begin{equation}
\label{almostorthbessel}
\frac{[fhH]_{D_R}^2}{[h^2H^2]_{D_R}} + \frac{[f(y\cdot\mathbf n)]^2_{D_R}}{[(y\cdot \mathbf n)^2]_{D_R}}
\leq \frac{1}{1-\epsilon} [f^2]_{D_R}\text.
\end{equation}
By combining \eqref{almostorthineq1}~and~\eqref{almostorthbessel} and rearranging, we obtain the desired result.
\end{enumerate}
\end{proof}
\subsection{Proof of the main theorem}
\label{proofsection}
The main new content of this section is the following lemma, which constructs unstable variations assuming spectral information about~$L$.
\begin{lemma}
\label{unstablevariations}
Let $\Sigma$ be an immersed, rotationally symmetric, self-shrinking hypersurface in~$\bbR^{n+1}$ with polynomial volume growth, satisfying $\mu_1(0, \Sigma) < -1$ and $\mu_1(1, \Sigma) < -1/2$. For $R$ sufficiently large (depending on~$\Sigma$), the following holds: 

Let $f_0 \in C^\infty_0(0, D_R)$ and $f_1 \in C^\infty_0(1, D_R)$ be eigenfunctions corresponding to $\mu_1(0, D_R)$ and $\mu_1(1, D_R)$, respectively. Let also $g_1 \in C^\infty_{0, \mathrm{sin}}(1, D_R)$ be the lowest eigenfunction of $L$ in $C^\infty_{0, \mathrm{sin}}(1, D_R)$ (i.e. $g_1 = f_1 \tan \theta$). Then any nonzero linear combination $f = \alpha f_0 + \beta f_1 + \gamma g_1$ represents an unstable variation of $\Sigma$ in the sense of Definition \ref{unstabledefn}.
\end{lemma}
\begin{proof}Fix an $R \geq R_0$, with $R_0$ to be chosen later, and let $y \in \bbR^{n+1}$ and $h, \alpha, \beta, \gamma \in \bbR$ be arbitrary, with the restriction that $\alpha$, $\beta$, and $\gamma$ not all be zero.

We take $f = \alpha f_0 + \beta f_1 + \gamma f_{1, s}$ as in the statement. By Proposition~\ref{secondvariationformula}, our task is to show that
\begin{equation}
\label{toshowsecondvariation}
\left[ -f L f + 2 f h H - h^2 H^2 + f y \cdot \mathbf n - \frac 12 (y \cdot \mathbf n)^2\right]_{D_R} < 0\text.
\end{equation}
Note that $f_0$, $f_1$, $f_{1, s}$ depend on $R$, so we will need to take care to ensure that the conditions we impose on $R$ do not depend on $f$, nor on $y$, $h$, $\alpha$, $\beta$, or $\gamma$.

Write $y = y_0 + y_1 + y_{1,s}$, where $y_1$ is parallel to~$e_n$, $y_{1,s}$ is parallel to~$e_{n+1}$, and $y_0$ is orthogonal to both $e_n$ and $e_{n+1}$. By Proposition~\ref{knownfuncsspecial}, it follows that $y_0 \cdot \mathbf n \in C^\infty(0, \Sigma)$, $y_1\cdot \mathbf n \in C^\infty(1, \Sigma)$ and $y_{1,s}\cdot \mathbf n \in C^\infty_{\mathrm{sin}}(1, \Sigma)$. Furthermore, we also have $H \in C^\infty(0, \Sigma)$.

We thus have decompositions of $f$ and $y \cdot \mathbf n$ into components corresponding to the three spaces $C^\infty(0, \Sigma$), $C^\infty(1, \Sigma)$, and $C^\infty_{\mathrm{sin}}(1, \Sigma)$:
\begin{align*}
f &= \alpha f_0 + \beta f_1 + \gamma g_1 \\
y\cdot\mathbf n &= y_0 \cdot \mathbf n + y_1 \cdot \mathbf n + y_{1, s} \cdot \mathbf n.
\end{align*}
By substituting these formulas for $f$ and $y\cdot \mathbf n$, the left-hand side of~\eqref{toshowsecondvariation} can be split up into components corresponding to those three spaces:
\begin{equation*}
\left[ -f L f + 2 f h H - h^2 H^2 + f y\cdot \mathbf n - \frac 12 (y \cdot \mathbf n)^2\right]_{D_R} = A +  B + C
\end{equation*}
where
\begin{align*}
A &= \left[-\alpha^2 f_0 Lf_0 + 2\alpha f_0 hH - h^2 H^2 + \alpha f_0y_0 \cdot \mathbf n
- \frac{(y_0 \cdot \mathbf n)^2}{2}\right]_{D_R}\text,\\
B &= \left[-\beta^2 f_1 Lf_1 + \beta f_1y_1\cdot \mathbf n
- \frac{(y_1\cdot\mathbf n)^2}{2}\right]_{D_R}\text{, and} \\
C &= \left[-\gamma^2 g_1 Lg_1 + \gamma g_1y_{1,s}\cdot \mathbf n 
- \frac{(y_{1,s} \cdot\mathbf n)^2}{2}\right]_{D_R}\text,
\end{align*}
and where all the other cross terms vanish because of orthogonality between the spaces.

Note that $-f_1 Lf_1 = \mu_1(1, D_R) f_1^2$.
Since $\mu_1(1, \Sigma) < -1/2$, we can pick $R_1$ such that $\mu_1(1, D_{R_1}) < -1/2$. Assuming that $R_0 \geq R_1$, we then have $R \geq R_1$, and $\mu_1(1, D_R) < -1/2$. Thus
\begin{equation*}
B \leq \left[-\frac 12 \beta^2 f_1^2 + \beta f_1 y_1 \cdot n - \frac{(y_1 \cdot \mathbf n)}{2}\right]_{D_R}\text,
\end{equation*}
and Lemma~\ref{variationpartialbounds}(i) now shows that $B \leq 0$. Equality holds only when $\beta = 0$. 

The same argument shows that $C \leq 0$, with equality only when $\gamma = 0$.

As for $A$, we have $-f_0 L f_0 = \mu_0(1, D_R) f_0^2$, so
\begin{equation*}
A = \left[\alpha^2 \mu_0(1, D_R) + 2\alpha f_0 h H - h^2 H^2 + \alpha f_0 y_0 \cdot \mathbf n - \frac{(y_0 \cdot \mathbf n)^2}{2}\right]_{D_R}\text.
\end{equation*}
To show that $A \leq 0$, we will split into two cases:
\begin{description}
\item[Case 1: $\mu_1(0, \Sigma) = -\infty$.] In this case, we can choose $R_2$ such that $\mu_1(0, D_{R_2}) < -3/2$. Assuming that $R \geq R_2$, we get $\mu_0(1, D_R) \leq -3/2$, and Lemma \ref{variationpartialbounds}(ii) then shows that $A \leq 0$, with equality only if $\alpha = 0$.
\item[Case 2: $-\infty < \mu_1(0, \Sigma) < -1$.] 
In this case, we first choose a large $R_3$ such that $\mu_1(0, D_{R_3}) < -1$, and then an $\epsilon > 0$ such that
\begin{equation}
\label{epsilondistance}
\mu_1(0, D_{R_3}) < -  \frac{1}{1-\epsilon}\text.
\end{equation}
Since $\mu_1(0, \Sigma) > -\infty$, Lemma~\ref{timespaceuniform} shows that we can find an $R_2 \geq R_3$, independent of $y$, such that
\begin{equation*}
\abs{\left[H(y\cdot\mathbf n)\right]_{D_{R}}} \leq \epsilon \left[H\right]_{D_{R}}^{1/2}\left[(y\cdot \mathbf n)^2\right]_{D_{R}}^{1/2}\text,
\end{equation*}
whenever $R \geq R_2$. By monotonicity of $\mu_1(0, D_R)$, \eqref{epsilondistance} also gives
\begin{equation}
\mu_1(0, D_{R}) < -\frac{1}{1-\epsilon}\text,
\end{equation}
and we get
\begin{equation*}
A \leq \left[-\frac{1}{1-\epsilon} f_0^2 + 2\alpha f_0 hH - hH^2 + \alpha f_0 y_0 \cdot n - \frac{(y_0 \cdot \mathbf n)^2}{2}\right]_{D_{R}}.
\end{equation*}
Now, we can apply Lemma~\ref{variationpartialbounds}(iii) to show that $A \leq 0$. Again, equality is only possible if $\alpha = 0$.
\end{description}
Assuming that we initially picked $R_0 = \max\{R_2, R_3\}$, we have shown that $A \leq 0$, $B \leq 0$, and $C \leq 0$, with simultaneous equality only happening when $\alpha = \beta = \gamma = 0$. Since $f$ is assumed non-zero, we in fact get $A + B + C < 0$. This holds for all possible $y$ and $h$, and thus $f$ represents an unstable variation.
\end{proof}
\begin{lemma}
\label{hchangessign}
Suppose $\Sigma$ is an immersed, rotationally symmetric self-shrinking hypersurface with polynomial volume growth. Write $e_n \cdot \mathbf n = u(x') \cos(\theta)$, where $u$ is a function on~$\Sigma'$, and assume that $u$ changes sign in the interior of~$\Sigma'$. Then $H$ changes sign on $\Sigma$.
\end{lemma}
\begin{proof}
Suppose not, i.e.\ that $H \geq 0$ for a suitable choice of normal direction. If $\Sigma$ is embedded, Theorem~10.1 of~\cite{cm} gives that $\Sigma$ is a sphere, a cylinder, or a plane. For these shrinkers, $u$ does not change sign, which rules them out. 

Since our $\Sigma$ is not necessarily embedded, we need to say a little more. The proof of Theorem~10.1 in~\cite{cm} (see also the remarks below Theorem~0.17 in the same work) reveals that embeddedness is only used to rule out shrinkers of the form $\Sigma = C \times \bbR^{n-1}$, where $C$ is one of the non-embedded self-shrinking curves discovered by Abresch and Langer in~\cite{abreschlanger}. However, these shrinkers are only rotationally symmetric if $n-1 \geq 2$ and the rotation axis is orthogonal to the $\bbR^{n-1}$ part. In that case, we have $e_n \cdot \mathbf n \equiv 0$, and so $u$ does not change sign. This finishes the proof even in the immersed case.
\end{proof}
Finally, we are in a position to prove the main theorem:
\begin{proof}[Proof of Theorem~\ref{maintheorem}]
Consider the functions $H$ and $e_n \cdot \mathbf n$ on $\Sigma$. By Proposition~\ref{knowneigenfunctions}, they are eigenfunctions of~$L$ of eigenvalues $-1$ and $-1/2$, respectively. By Proposition~\ref{knownfuncsspecial}, we have $H \in C^\infty(0, \Sigma)$ and $e_n \cdot \mathbf n \in C^\infty(1, \Sigma)$. Furthermore, since $\abs H = \abs{x \cdot \mathbf n} / 2 \leq \abs{x}/2$ and $\abs{e_n \cdot \mathbf n} \leq 1$, both functions have polynomial growth. We can write $e_n \cdot \mathbf n = u(x') \cos(\theta)$, where $u$ changes sign on $\Sigma'$ by assumption. By Lemma~\ref{hchangessign}, $H = v \cos(0)$ changes sign too. Theorem~\ref{bottomofspectrumtest} now tells us that $\mu_1(0, \Sigma) < -1$ and $\mu_1(0, \Sigma) < -1/2$, and Lemma~\ref{unstablevariations} shows that the $F$-index of~$\Sigma$ is at least 3.
\end{proof}
\begin{proof}[Proof of Corollary~\ref{maincor}]
Let $C \subset \bbR_{\geq 0} \times \bbR$ be the image of~$\Sigma$ under the quotient map $\bbR^{n+1} \rightarrow \bbR^{n+1}/SO(n)$. By the same reasoning as in the proof of Lemma~\ref{sigmaprimeprops}, $C$ is an immersed curve whose interior is connected.

Let $e_R$ denote the first standard basis vector in $\bbR_{\geq 0} \times \bbR$. We will first show that $e_R \cdot \mathbf n$ changes sign on~$C$. If $\Sigma$ is embedded, then by the classification in Theorem~2 of~\cite{mollerkleene}, $C$ is a closed curve, and thus $e_R \cdot \mathbf n$ must change sign on $C$. On the other hand, if $\Sigma$ is not embedded, then $C$ is not embedded either. Since $C$ is connected, one section of $C$ must be a closed loop, and again, it follows that $e_R \cdot \mathbf n$ must change sign on $C$.

Now, $C$ can be identified with a particular one-dimensional slice of $\Sigma$, on which $e_r \cdot \mathbf n$ is equal to the function $e_R \cdot \mathbf n$ on $C$. Thus $e_r \cdot \mathbf n$ changes sign on $\Sigma$, and we can apply Theorem~\ref{maintheorem}.
\end{proof}
\subsection{Entropy index}
\label{entropyproof}
In this section, we will prove Corollary~\ref{entropycor}, which says that under the hypotheses of either Theorem~\ref{maintheorem} or Corollary~\ref{maincor}, the entropy index is at least 3.

First of all, we specify what is meant by this conclusion. Being a supremum of $F$-functionals, the entropy is not twice differentiable even in a formal sense, so it is not immediately clear what the entropy index means. We make the following definition:
\begin{defn}
\label{entropyindexdefn}Suppose that $\Sigma$ is an immersed shrinker in $\bbR^{n+1}$. We say that $\Sigma$ has \emph{entropy index at least $k$} if we can find a linear space~$V$, consisting of smooth, normal vector fields to~$\Sigma$, with $\dim V = k$, satisfying the following property: If $X$ is any nonzero vector field in~$V$, and $\Sigma_s$ is any variation of~$\Sigma$ with variation vector field~$V$, then there is some $\epsilon > 0$ such that $\lambda(\Sigma_s) < \lambda(\Sigma)$ for all $s \in (-\epsilon, \epsilon) \setminus \{0\}$.
\end{defn}

In the situation when the shrinker $\Sigma$ does not split off a line isometrically, results in~\cite{cm} give Corollary~\ref{entropycor} as a direct consequence of Theorem~\ref{maintheorem} and Corollary~\ref{maincor}. In the case when $\Sigma$ does split off a line, we will supply a little more reasoning, in the form of the two following lemmas.
\begin{lemma}
\label{splitofflinehypotheses}Suppose that $\Sigma$ is an immersed, self-shrinking hypersurface in $\bbR^{n+1}$. Write $\Sigma$ as an isometric product $\tilde \Sigma \times \bbR^m$. If $\Sigma$ satisfies the hypotheses of either Theorem~\ref{maintheorem} or Corollary~\ref{maincor}, then so does $\tilde \Sigma$.
\end{lemma}
\begin{proof}
We decompose $\bbR^{n+1}$ as $V_1 \oplus V_2$, where $\dim V_1 = m$, and $V_1$ consists of the same directions as the factor $\bbR^{m}$ in the decomposition $\Sigma = \tilde \Sigma \times \bbR^m$. In this way, we can view $\tilde \Sigma$ as an immersed hypersurface in $V_2$, and it is easy to see that $\tilde \Sigma$ is a shrinker.

We now split into two cases:
\begin{description}
\item[Case 1: $\Sigma$ satisfies the hypotheses of Corollary~\ref{maincor}.] $\Sigma$ is not a sphere, cylinder or a plane, and therefore, neither is $\tilde \Sigma$. Furthermore, $\Sigma$ is invariant under a subgroup $SO(n) \subset SO(n+1)$. Letting $W$ denote the $n$-dimensional hyperplane in $\bbR^{n+1}$ on which this $SO(n)$ acts, a simple dimension count shows that $\dim W \cap V_2 \geq \dim V_2 - 1$. It follows that $\tilde \Sigma$, as an immersed hypersurface in $V_2$, is invariant under a subgroup $SO(n-m)$ of $SO(n+1-m)$. Thus, $\tilde \Sigma$ also satisfies the hypotheses of Corollary~\ref{maincor}.
\item[Case 2: $\Sigma$ satisfies the hypotheses of Theorem~\ref{maintheorem}.] In particular, this means that we can find a decomposition $\bbR^{n+1} = W_1 \oplus W_2$, where $\dim W_1 = 2$, and where $\Sigma$ is invariant under the group $SO(2)$ acting as rotations in $W_1$.

Assume, as a hypothesis for contradiction, that $W_1$ not perpendicular to $V_1$. Let $v \in V_1$ be a vector whose orthogonal projection onto $W_1$, which we will call $\Pi v$, is non-zero. Let also $R$ denote rotation by an angle $\pi$ in $W_1$. By assumption, $\Sigma$ is invariant under both $R$ and translation by $c v$, for any $c \in \bbR$. Therefore, it must be invariant under the composition
\begin{equation*}
x \mapsto R^{-1}\left(R\left(x + \frac 12c v\right) - \frac 12 c v\right)\text.
\end{equation*}
A little algebra reveals that the right-hand side simplifies to $x + c\Pi v$; in other words, $\Sigma$ is invariant under translations by $c \Pi v$ as well. Since $\Pi v$ is nonzero, we can show that $\Sigma$ is invariant under translations by any vector in $W_1$ by considering conjugations of translations by $c \Pi v$ and rotations in $W_1$. This means that the tangent plane $T_x \Sigma$ of $\Sigma$ must be parallel to $W_1$ for all $x \in \Sigma$., which in turn implies that $e_r \cdot \mathbf n = 0$ everywhere (where $e_r$ is a vector parallel to $W_1$ which points radially outward), which contradicts the second hypothesis of Theorem~\ref{maintheorem}. Thus we see that $W_1$ must be perpendicular to $V_1$, or in other words, contained in $V_2$.

From this, it now follows easily that $\tilde \Sigma$ is invariant under this copy of $SO(2)$, thought of as acting on $V_2$. Furthermore, $e_r \cdot \mathbf n$ must then change sign on $\tilde \Sigma$, because it does on $\Sigma$. Thus, we conclude that $\tilde \Sigma$ satisfies the hypotheses of Theorem~\ref{maintheorem}.
\end{description}
\end{proof}
\begin{lemma}
\label{transferentropyindex}Suppose that $\Sigma$ is an immersed, self-shrinking hypersurface in $\bbR^{n+1}$. Write $\Sigma$ as an isometric product $\tilde \Sigma \times \bbR^m$. If $\tilde \Sigma$ has entropy index at least $k$, then so does does $\Sigma$.
\end{lemma}
\begin{proof}
Given a variation $\tilde \Sigma_s$ of $\tilde \Sigma$, we can extend it to a variation of $\Sigma$ by taking $\Sigma_s = \tilde \Sigma_s \times \bbR^m$. It is fairly easy to see that $\lambda(\Sigma_s) = \lambda(\tilde \Sigma_s)$ (see the proof of Theorem~0.12 in Section~11.2 of~\cite{cm}).

This way, any variation of $\tilde \Sigma$ which is unstable for the entropy (in the sense of Definition~\ref{entropyindexdefn}) gives an unstable variation of~$\tilde \Sigma$. Since $\tilde \Sigma$ has a $k$-dimensional vector space worth of such variations, so does $\Sigma$.
\end{proof}
\begin{proof}[Proof of Corollary~\ref{entropycor}]Assume first that $\Sigma$ does not split off a line isometrically. By Theorem~\ref{maintheorem} or Corollary~\ref{maincor}, we can find a subspace~$V$, of dimension 3, consisting of $F$-unstable variation vector fields. Suppose that $X$ is a vector field in~$V$, and that $\Sigma_s$ is a variation with variation vector field~$X$. By Theorem~0.15 of~\cite{cm} (or more precisely, by the proof, which is found in section~7.2 of that work), we have $\lambda(\Sigma_s) < \lambda(\Sigma)$ for $s$ nonzero and sufficiently small. Since this holds for any $X \in V$, this shows that $\Sigma$ has entropy index at least 3.

If $\Sigma$ does split off a line, Theorem~0.15 of~\cite{cm} does not apply. However, we can write $\Sigma = \tilde \Sigma \times \bbR^m$, where $\tilde \Sigma$ does not split off a further line. By Lemma~\ref{splitofflinehypotheses}, $\tilde \Sigma$ also satisfies the hypotheses of either Theorem~\ref{maintheorem} or Corollary~\ref{maincor}. By the previous reasoning, this means that $\tilde \Sigma$ has entropy index at least 3, and by Lemma~\ref{transferentropyindex}, $\Sigma$ also has entropy index at least 3.
\end{proof}
\section{Regularity}
The analysis of the preceding sections assumes that the shrinker is smooth. In this section, we will make some remarks on the connection between index and regularity. 

The main theorem is this:
\begin{thm}
Let $\Sigma$ be an $n$-dimensional integral varifold in $\bbR^{n+1}$ which is stationary with respect to the functional $F_{0,1}$, and suppose that $\Sigma$ satisfies the volume bound $\vol(B_r(x) \cap \Sigma) \leq V r^n$ for all $0 < r < 1$. Assume that $n \leq 6$, and also that $\Sigma$ satisfies the $\alpha$-Structural Hypothesis of Wickramasekera(\cite{wickramasekera}) for some $\alpha \in (0, 1)$.Then the $F$-index of the regular part of $\Sigma$ is greater than or equal to the number of singular points of $\Sigma$.
\end{thm}
\begin{proof}
The proof is adapted from Section~12 of~\cite{cm}.

First of all, if $U$ is an open set in $\bbR^{n+1}$, we will call $\Sigma \measurerestr U$ \emph{$F$-stable} if any function $\phi$ which has compact support in $\reg \Sigma \cap U$ is a stable variation in the sense of Definition~\ref{unstabledefn}.

Assume that $\Sigma \measurerestr U$ is $F$-stable. Given $\epsilon > 0$, Lemma~12.7 of~\cite{cm} shows that
\begin{itemize}
\item $\Sigma$ is a stationary varifold in the Riemannian manifold $U \subset \bbR^{n+1}$ with respect to the metric $g_{ij} = e^{-\abs{x}^2/4} \delta_{ij}$.
\item For $r_0$ small enough (depending on $\epsilon$), if $\phi$ is a function supported in some ball $B_{r_0}(x_0) \cap \reg \Sigma$ satisfies
\begin{equation*}
\int_{\Sigma} \abs{A_g}^2_g \phi^2 d\vol_g \leq (1 + \epsilon) \int_\Sigma \abs{\nabla_g \phi}^2_g d\vol_g\text.
\end{equation*}
\end{itemize}
By the version of Wickramasekera's main regularity theorem (\cite{wickramasekera}) formulated as Proposition~12.25 in~\cite{cm}, in fact $\Sigma \measurerestr U$ is smooth.

Suppose now that we have $k$ distinct singular points $x_1, \ldots, x_k$ in $\Sigma$. Let $U_1, \ldots, U_k$ be disjoint open neighborhoods of $x_1, \ldots, x_k$ in $\bbR^{n+1}$. For each $i$, since $\Sigma \measurerestr U_i$ has a singularity, the preceding argument shows that we must have a function $\phi_i$ with support in $\reg \Sigma \cap U$ such that $\phi_i$ represents an unstable variation of $\Sigma$. Since these $\phi_i$ have disjoint supports, any linear combination $\sum_{i=1}^k a_i \phi_i$ will also represent an unstable variation. It follows that $\Sigma$ has $F$-index at least $k$.
\end{proof}

\bibliographystyle{alpha}
\bibliography{findex2}

\begin{thebibliography}{{McG}15}

\bibitem[AL86]{abreschlanger}
U.~Abresch and J.~Langer.
\newblock The normalized curve shortening flow and homothetic solutions.
\newblock {\em J. Differential Geom.}, 23(2):175--196, 1986.

\bibitem[ALW14]{andrewsliwei}
Ben Andrews, Haizhong Li, and Yong Wei.
\newblock {$\mathcal F$}-stability for self-shrinking solutions to mean
  curvature flow.
\newblock {\em Asian J. Math.}, 18(5):757--777, 2014.

\bibitem[Ang92]{angenenttori}
Sigurd~B. Angenent.
\newblock Shrinking doughnuts.
\newblock In {\em Nonlinear diffusion equations and their equilibrium states, 3
  ({G}regynog, 1989)}, volume~7 of {\em Progr. Nonlinear Differential Equations
  Appl.}, pages 21--38. Birkh\"auser Boston, Boston, MA, 1992.

\bibitem[CM12]{cm}
Tobias~H. Colding and William~P. Minicozzi, II.
\newblock Generic mean curvature flow {I}: generic singularities.
\newblock {\em Ann. of Math. (2)}, 175(2):755--833, 2012.

\bibitem[CMP15]{coldingminicozzipedersen}
Tobias~Holck Colding, William~P. Minicozzi, and Erik~Kj{\ae}r Pedersen.
\newblock Mean curvature flow.
\newblock {\em Bull. Amer. Math. Soc. (N.S.)}, 52(2):297--333, 2015.

\bibitem[DK13]{drugankleene}
G.~{Drugan} and S.~J. {Kleene}.
\newblock {Immersed self-shrinkers}.
\newblock {\em ArXiv e-prints}, June 2013.

\bibitem[Dru15]{drugan}
Gregory Drugan.
\newblock An immersed {$S^2$} self-shrinker.
\newblock {\em Trans. Amer. Math. Soc.}, 367(5):3139--3159, 2015.

\bibitem[DX13]{dingxin}
Qi~Ding and Y.~L. Xin.
\newblock Volume growth, eigenvalue and compactness for self-shrinkers.
\newblock {\em Asian J. Math.}, 17(3):443--456, 2013.

\bibitem[Eva10]{evans}
Lawrence~C. Evans.
\newblock {\em Partial differential equations}, volume~19 of {\em Graduate
  Studies in Mathematics}.
\newblock American Mathematical Society, Providence, RI, second edition, 2010.

\bibitem[GT83]{gilbargtrudinger}
David Gilbarg and Neil~S. Trudinger.
\newblock {\em Elliptic partial differential equations of second order}, volume
  224 of {\em Grundlehren der Mathematischen Wissenschaften [Fundamental
  Principles of Mathematical Sciences]}.
\newblock Springer-Verlag, Berlin, second edition, 1983.

\bibitem[Hui90]{huisken}
Gerhard Huisken.
\newblock Asymptotic behavior for singularities of the mean curvature flow.
\newblock {\em J. Differential Geom.}, 31(1):285--299, 1990.

\bibitem[Ilm95]{ilmanensingularities}
Tom Ilmanen.
\newblock Singularities of mean curvature flows of surfaces.
\newblock preprint, 1995.

\bibitem[KM14]{mollerkleene}
Stephen Kleene and Niels~Martin M{\o}ller.
\newblock Self-shrinkers with a rotational symmetry.
\newblock {\em Trans. Amer. Math. Soc.}, 366(8):3943--3963, 2014.

\bibitem[LL15]{leelue}
Yng-Ing Lee and Yang-Kai Lue.
\newblock The stability of self-shrinkers of mean curvature flow in higher
  co-dimension.
\newblock {\em Trans. Amer. Math. Soc.}, 367(4):2411--2435, 2015.

\bibitem[{McG}15]{mcgrath}
P.~{McGrath}.
\newblock {Closed Mean Curvature Self-Shrinking Surfaces of Generalized
  Rotational Type}.
\newblock {\em ArXiv e-prints}, July 2015.

\bibitem[Rig96]{rigger}
Ralf~Otto Rigger.
\newblock Selbst{\"a}hnliche {L}{\"o}sungen des mittleren
  {K}rümmungsflu{\ss}es.
\newblock Diplomarbeit, Eberhard Karls Universit{\"a}t T\"ubingen, 1996.

\bibitem[Urb90]{urbano}
Francisco Urbano.
\newblock Minimal surfaces with low index in the three-dimensional sphere.
\newblock {\em Proc. Amer. Math. Soc.}, 108(4):989--992, 1990.

\bibitem[Wic14]{wickramasekera}
Neshan Wickramasekera.
\newblock A general regularity theory for stable codimension 1 integral
  varifolds.
\newblock {\em Ann. of Math. (2)}, 179(3):843--1007, 2014.

\end{thebibliography}
\end{document}